\documentclass[12pt]{amsart}

\usepackage{amsmath}
\usepackage{amsbsy}
\usepackage{amssymb}
\usepackage{amscd}
\usepackage{eq}
\usepackage[dvips]{graphicx, color}

\newcommand{\bbc}{{\mathbb C}}

\newcommand{\bbn}{{\mathbb N}}

\newcommand{\bbp}{{\mathbb P}}
\newcommand{\bbq}{{\mathbb Q}}
\newcommand{\bbr}{{\mathbb R}}

\newcommand{\bbz}{{\mathbb Z}}

\newcommand{\be}{{\beta}}

\newcommand{\gam}{{\gamma}}

\newcommand{\Gam}{{\Gamma}}

\newcommand{\gB}{{\mathfrak B}}

\newcommand{\gI}{{\mathfrak I}}

\newcommand{\gJ}{{\mathfrak J}}

\newcommand{\gS}{{\mathfrak S}}
\newcommand{\gt}{{\mathfrak t}}

\newcommand{\ch}{{\operatorname{ch}}\,}

\newcommand{\aff}{{\operatorname {Aff}}}

\newcommand{\m}{{\operatorname {M}}}

\newcommand{\im}{{\operatorname {Im}}}

\newcommand{\gal}{{\operatorname{Gal}}}
\newcommand{\gl}{{\operatorname{GL}}}

\newcommand{\spl}{{\operatorname{SL}}}

\newcommand{\sst}{{\operatorname{ss}}}

\newcommand{\rep}{representation}

\newcommand{\pv}{prehomogeneous vector space}

\newcommand{\Z}{\bbz}
\newcommand{\Q}{\bbq}
\newcommand{\R}{\bbr}
\newcommand{\C}{\bbc}
\newcommand{\p}{\bbp}

\newcommand{\kableadd}%
{Department of Mathematics\\ Cornell University\\
Ithaca NY 14853}

\newcommand{\sub}{\subset}

\newcommand{\ccd}{,\ldots,}

\makeatletter
\def\varddots{\mathinner{
\mkern1mu%
 \raise\p@\hbox{.}\mkern2mu%
 \raise4\p@\hbox{.}\mkern2mu%
 \raise7\p@\vbox{\kern7\p@\hbox{.}}%
\mkern1mu}}
\makeatother

\theoremstyle{plain}
\newtheorem{thm}{Theorem}[section]

\newtheorem{prop}[thm]{Proposition}

\newtheorem{algo}[thm]{Algorithm}

\theoremstyle{definition}

\newtheorem{cond}[thm]{Condition}

\theoremstyle{remark}
\usepackage{mathrsfs,bbm}

\newcommand{\weyl}{\mathbb W}
\newcommand{\Conv}{\operatorname{Conv}}
\newcommand{\Span}{\operatorname{Span}}
\newcommand{\semi}{\mathrm {ss}}
\newcommand{\coorde}{{\mathbbm a}}
\newcommand{\diag}{\operatorname{diag}}
\newcommand{\bbmp}{\mathbbm p}

\begin{document}

\address[K. Tajima]
{National Institute of Technology, Sendai College, Natori Campus, 
48 Nodayama, Medeshima-Shiote, Natori-shi, Miyagi, 981-1239, Japan}
\email{kazuaki.tajima.a8@tohoku.ac.jp}

\address[A. Yukie]
{Department of Mathematics, Graduate School of Science, 
Kyoto University, Kyoto 606-8502, Japan}
\email{yukie.akihiko.7x@kyoto-u.ac.jp}
\thanks{The second author was partially supported by 
Grant-in-Aid (C) (17K05169)\\}


\title{On the GIT stratification of prehomogeneous vector spaces I}
\author{Kazuaki Tajima}
\author{Akihiko Yukie}
\maketitle

\begin{abstract}
We determine the set which parametrizes the GIT stratification 
for four \pv s in this paper.  
\end{abstract}

\section{Introduction}
\label{sec:introduction}

This is part one of a series of four papers. 
Let $k$ be a perfect field. 
In this series of papers, we determine the GIT (geometric invariant theory) 
stratification of the following \pv s over $k$.

\vskip 5pt

(1) $G=\gl_3\times \gl_3\times \gl_2$, 
$V=\aff^3\otimes \aff^3\otimes \aff^2$.

(2) $G=\gl_6\times \gl_2$, 
$V=\wedge^2 \aff^6\otimes \aff^2$.

(3) $G=\gl_5\times \gl_4$,  
$V=\wedge^2 \aff^5\otimes \aff^4$.

(4) $G=\gl_8$, $V=\wedge^3 \aff^8$.  

\vskip 5pt
If the base field is $\C$ then orbits of (1)--(4) 
have been determined in \cite[pp. 385--387]{kimu}, 
\cite[pp. 456,457]{kimura-orbits}, 
\cite{ozekic}, \cite{gurevich} (see \cite[p.19]{ozekib} also)
respectively. 

The notion of GIT stratification 
was established by Ness, Kempf and Kirwan in 
\cite{kene}, \cite{kempf}, \cite{ness}, \cite{kirwan}. 
This notion will be reviewed in Section \ref{sec:git-stratification}. 
If the base field $k$ is algebraically closed then 
the GIT stratification gives us the orbit decomposition. 
The advantage of the GIT stratification is that it answers   
the rationality question of orbits. 
For the rationality of the GIT stratification, see 
\cite{tajima-yukie} (if the group is split, 
the rationality follows easily from \cite{kempf}). 
For the \pv s (1)--(4), 
we determine all orbits rationally over $k$. Moreover, 
the inductive structure of strata is guaranteed. 
Some smaller \pv s has been considered in \cite{ishimoto} by naive method. 
%Note that the case (1) follows from the case (3).  However, 
%we included the case (1) since it is a rather interesting case. 
%We shall point out why the case (1) follows from the case (3) 
%at the end of Section \ref{sec:git-stratification}. 

We refer to parts of this series of papers as Part I--Part IV. 
The GIT stratification is parametrized by a certain finite set $\gB$  
(see Section \ref{sec:git-stratification}). 
This set $\gB$ is combinatorially defined and so it is possible to 
determine $\gB$ by computer computations.  The purpose of this part 
is to carry out the computer computations to determine $\gB$ for 
(1)--(4). 
 
%Note that 
The stratum corresponding to $\be\in\gB$ could be the empty set. 
So it is important to 
determine which strata $S_{\be}$ are non-empty. 
%In fact many $S_{\be}$'s are the empty set.  
We carry this out and determine rational orbits in $S_{\be}$ 
for (1), (2) in Part II \cite{tajima-yukie-GIT2}, 
(3) in Part III \cite{tajima-yukie-GIT3} 
($\ch(k)\not=2$ is assumed in \cite{tajima-yukie-GIT3} 
to determine rational orbits in $S_{\be}$)
and (4) in Part IV \cite{tajima-yukie-GIT4}.  

The cardinality of the set $\gB$ for (1)--(4) is given in the following 
theorem. 
\begin{thm}
\label{thm:main}  
The cardinality of the set $\gB$ for the \pv s (1)--(4) is 
$49,81,292,183$ respectively.  
\end{thm}

We list elements of $\gB$ for the 
\pv s (1)--(4) in Sections \ref{sec:output1}, 
\ref{sec:output2}, \ref{sec:output3}, \ref{sec:output4}
respectively.  The numbers of non-empty strata are 
$16,13,61$ for the cases (1), (2), (3) respectively 
(see \cite{tajima-yukie-GIT2}, \cite{tajima-yukie-GIT3} 
(no assumption on $\ch(k)$ for this part). 
Note that in \cite{kimu}, \cite{kimura-orbits}, 
$V^{\sst},\{0\}$ are counted and so the numbers of orbits 
are $18,15,63$ respectively. We expect to have $21$ non-empty strata 
for the case (4). 

The \pv s (1)--(4) are rather important \pv s with 
interesting arithmetic interpretations of rational 
orbits (see \cite{wryu}, \cite{yukiej}).  The determination 
of the GIT stratification may have applications to some 
fields in number theory such as the zeta function theory.

The organization of this part is as follows. 
We review the notion of GIT stratification in 
Section \ref{sec:git-stratification}.  
We explain the outline of the computer program in 
Section \ref{sec:outline-program}. 
We shall use multiple arrays in  the computer program. 
We have to be careful not to use too much memory space
and we have to go back and forth between multiple arrays 
and a single array. We discuss the 
combinatorial problem regarding the lexicographical 
order of combinations in Section \ref{sec:lexic-order-comb}. 

Roughly speaking, we consider the set of weights of 
the \rep s (1)--(4) and find the closest point to the 
origin from the convex hull of each finite subset
of the set of weights. Since the Weyl group acts on the set
of finite subsets of the set of weights, 
we first find a set of representatives. 
For this part, we do not have to worry about 
the possibility of overflow and we can use a 
computer language such as ``C''. 
After reducing the number of cases, we make a certain matrix 
and find the rref (reduced row echelon form) for each case.  
For this part, we have to use a computer language such as ``MAPLE''
with no restriction of digits.  
We explain some details of the computer programs in 
Section \ref{sec:programs}. 
We list outputs of our computer program in Sections 
\ref{sec:output1}--\ref{sec:output4}.

The authors would like to thank the referees for helpful comments and suggestions.

\section{GIT stratification}
\label{sec:git-stratification}

In this section we  briefly review the notion of 
GIT stratification.  Let $k$ be a perfect field and 
$\overline k$ its algebraic closure. 
If $X$ is a finite set, then $\# X$ will denote its cardinality.
The standard symbols $\Q$, $\R$, $\C$, $\Z$ and $\bbn$ 
will denote respectively
the fields of rational, real, complex numbers, 
the ring of rational integers and the set of non-negative integers. 
Let $\gS_n$ be the permutation group of $\{1\ccd n\}$. 

We denote the space of $n\times m$ matrices by $\m_{m,n}$, 
$\m_n=\m_{n,n}$ and
the group of $n\times n$ invertible matrices by $\gl_n$.
Obviously, $\m_n$ has an algebra structure. 
Let $\spl_n=\{g\in\gl_n \mid \det(g)=1\}$.
We denote the unit matrix of dimension $n$ 
by $I_n$.  We use the notation $\diag(g_1\ccd g_m)$ 
for the block diagonal matrix whose diagonal blocks are 
$g_1\ccd g_m$.

We are mainly interested in \pv s, but 
we first consider a more general situation.

Let $G$ be a connected reductive group, $V$ a finite dimensional 
representation of $G$ both defined over $k$. Since we only consider 
split reductive groups in this paper, we assume that 
$G$ is split. 
We assume that there is a connected split reductive subgroup $G_1$ of $G$, 
a split torus $T_0\sub Z(G)$ (the center of $G$), 
such that $T_0\cap G_1$ is finite and $G=T_0G_1$ 
as algebraic groups. We assume that there is a 
rational character $\chi$ of $T_0$ such that the action of 
$t\in T_0$ is given by the scalar multiplication by $\chi(t)$.

Let $(T_0\cap G_1)\sub T\sub G_1$ be a maximal split torus, 
$X_*(T),X^*(T)$ be the groups of one parameter subgroups 
(abbreviated as 1PS from now on) and the group of 
rational characters respectively. We put 
\begin{equation*}
\mathfrak {t}=X_*(T)\otimes \R, \;
\mathfrak {t}_{\Q}=X_*(T)\otimes \Q, \;
\mathfrak t^*=X^*(T)\otimes \R, \;
\mathfrak t^*_{\Q}=X^*(T)\otimes \Q.
\end{equation*}
Let $\weyl =N_{G}(T)/T$ be the Weyl group
of $G$. $\weyl$ acts on $\gt^*$ also. 

There is a natural pairing 
$\langle \;  ,\;  \rangle_T :X^*(T)\times 
X_*(T)\rightarrow \Z$  defined by 
$t^{\langle \chi ,\lambda \rangle_T}=\chi (\lambda (t))$ for 
$\chi \in X^*(T),\lambda \in X_*(T)$. 
This is a perfect paring (\cite[pp.113--115]{borelb}).

There exists an inner product 
$(\;  , \; )$
on $\mathfrak{t}$ which is invariant under the actions of $\weyl$
and the Galois group $\gal(\overline k/k)$. 
We may assume that this inner product is rational, i.e., 
$(\lambda,\nu)\in \Q$ for all 
$\lambda,\nu\in \mathfrak t_{\Q}$. 
Let $\|\; \|$ be the norm on $\mathfrak t$ defined 
by $(\;  , \; )$. We choose a Weyl chamber 
$\mathfrak{t}_{+}\subset \mathfrak{t}$ 
for the action of $\weyl$.

For $\lambda\in\mathfrak t$, 
let $\beta =\beta(\lambda)$ be the element of
$\mathfrak t^*$ such that 
$\langle \beta, \nu\rangle = (\lambda, \nu)$
for all $\nu\in \mathfrak t$. 
The map $\lambda\mapsto \beta(\lambda)$ is a bijection 
and we denote the inverse map by 
$\lambda=\lambda(\beta)$. 
There is a unique positive rational
number $a$ such that $a\lambda(\beta)\in X_*(T)$
and is indivisible. We use the notation 
$\lambda_{\beta}$ for $a\lambda(\beta)$.

Identifying $\mathfrak t$ with $\mathfrak t^*$ 
we have a $\weyl$-invariant inner product 
$(\;  , \;  )_{*}$ on $\mathfrak{t^*}$, 
the norm $\|\; \|_{*}$ determined by $(\;  , \;  )_{*}$ 
and a Weyl chamber $\mathfrak t^*_{+}$. 

Let $N=\dim V$. 
We choose a coordinate system 
$v=(v_1,\dots ,v_N)$ on $V$ by which $T$ acts diagonally. 
Let $\gamma_i \in \mathfrak t^{\ast }$ and $\coorde_i$ be 
the weight and the coordinate vector 
which corresponds to $i$-th coordinate. 
Let $\Gam=\{\gam_1\ccd \gam_N\}$. 
For a subset $\gI\subset \Gam$, 
we denote the convex hull of $\mathfrak I$ by $\Conv \mathfrak I$. 
Let $\p(V)$ be the projective space associated with $V$ 
and $\pi_V:V\setminus\{0\}\to \p(V)$ the natural map. 
For $\gI\subset \Gam$ 
such that $0\notin \Conv \mathfrak I$, let $\beta $ be the 
closest point of $\Conv\mathfrak  I$ to the origin. 
Then $\beta $ lies in $\mathfrak t^*_{\mathbb Q}$. 
Let $\mathfrak{B} $ be the set of all such $\beta $ 
which lies in $\mathfrak t_{+}^{\ast }$. 

We define
\begin{align*} 
& Y_{\beta }= \Span \{\coorde_i\,|\, (\gamma_i,\beta )_{*}
\geq (\beta ,\beta )_{*}\},\quad 
Z_{\beta }= \Span \{\coorde_i\,|\, (\gamma_i,\beta )_{*}
=(\beta ,\beta )_{*}\}, \\
& W_{\beta }=\Span \{ \coorde_i\,|\, (\gamma_i,\beta )_{*}
>(\beta ,\beta )_{*}\}
\end{align*}
where $\Span$ is the spanned subspace. 
Clearly $Y_{\beta}=Z_{\beta}\oplus W_{\beta}$.

If $\lambda $ is a 1PS of $G$, we define 
\begin{align*}
P(\lambda ) & = \left  \{p\in G\; \Big |\; 
\lim_{t\rightarrow 0}\lambda (t)p\lambda(t)^{-1} 
\; \textrm{exists} \right \},\; M(\lambda) = Z_G(\lambda) \;\text{(the centralizer)}, \\
U(\lambda ) & = \left  \{p\in G\; \Big |\; 
\lim_{t\rightarrow 0}\lambda (t)p\lambda(t)^{-1}= 1 \right \}. 
\end{align*} 
The group $P(\lambda )$ is a parabolic subgroup of 
$G$ (\cite[p.148]{Springer-LAG}) with Levi part $M(\lambda)$ 
and unipotent radical $U(\lambda)$. 
We put $P_{\beta}=P(\lambda_{\beta})$, 
$M_{\beta }=Z_G(\lambda_{\beta })$  and 
$U_{\beta }=U(\lambda_{\beta })$.

Let $\chi_{\beta}$ be the indivisible rational character of $M_{\beta}$
such that the restriction of $\chi_{\beta}^a$ 
to $T$ coincides with $b \beta$ for some positive integers $a,b$. 
We define $G_{\beta }=\{g\in M_{\beta }\,|\,\chi_{\beta}(g)=1\}^{\circ }$ 
(the identity component). Then $G_{\beta }$ acts on $Z_{\beta }$. 
Note that $M_{\beta }$ and $G_{\beta }$ are defined over $k$, and 
since $\langle \chi_{\beta},\lambda_{\beta}\rangle$ is a positive
multiple of $\|\beta\|_{*}$,
$M_{\beta}=G_{\beta}\im(\lambda_{\beta})$.
Moreover, if $\nu$ is any rational 1PS in $G_{\beta}$, 
$(\nu,\lambda_{\beta})=0$.

Let $\mathbb P(Z_{\beta })^{\semi}$ be the set of 
semi-stable points of $\mathbb P(Z_{\beta })$ with respect to 
the action of $G^1_{\be} \stackrel {\mathrm {def}}= G_{\beta}\cap G_1$. 
Since there is a difference between 
$Z_{\beta}$ and $\mathbb P(Z_{\beta })$, 
we remove appropriate scalar directions from 
$G_{\beta}$ to consider stability.  
For the notion of semi-stable points, see \cite{mufoki}. 
We regard $\mathbb P(Z_{\beta })^{\semi}$ as a subset of $\p(V)$. 
Put 
\begin{align*}
& Z_{\beta }^{\semi}=\pi_V^{-1}(\mathbb P(Z_{\beta })^{\semi }), \;
Y_{\beta }^{\semi}=\{(z,w)\,|\, z\in Z_{\beta}^{\semi},w\in W_{\beta}\}. 
%& \p(Y_{\beta})^{\semi}=\{\pi_V((z,w))\,|\, (z,w)\in Y_{\beta }^{\semi }\}.  
\end{align*}
We define $S_{\beta }=GY_{\beta }^{\semi }$. 
Note that $S_{\beta }$ can be the empty set. 
We denote the set of $k$-rational points of 
$S_{\beta}$, etc., by $S_{\beta \,k}$, etc.

%$\{(0,x,y)\,|\, x\in Z_{\beta }^{\semi },y\in W_{\beta }\}$.

The following theorem is COROLLARY 1.4 \cite[p.264]{tajima-yukie}. 
\begin{thm} 
\label{KKN} 
Suppose that $k$ is a perfect field. Then  we have 
\begin{align*} 
V_k\setminus \{0\} = V^{\semi } _k
\amalg \coprod_{\beta \in \mathfrak{B}} S_{\beta \, k}. 
\end{align*}
Moreover, $S_{\beta \,k}\cong G_{k}\times_{P_{\beta \,k}} Y_{\beta \,k}^{\semi }$. 
\end{thm}  

We call this stratification {\it the GIT stratification}. 
The importance of the above theorem is the rationality 
of the inductive structure of $S_{\beta}$. 
%The purpose of this part is to compute 
%$\gB$ for the \pv s in Introduction.  
%In parts II--IV, we determine the set of 
%$\be\in\gB$ such that $S_{\be}\not=\emptyset$. 
Obviously, we can use computer to determine $\gB$.

\section{Outline of the program}
\label{sec:outline-program}

In this section we explain the idea of the programming to 
compute the set $\gB$.

%We used C and MAPLE as computer languages to carry out the compujtations. 
%However, we explain the commputer program in the form 
%of a pseudo-program. Arrays in C (resp. MAPLE) starts from the 
%index $0$ (resp. $1$).  However, 
We assume that 
arrays start from the index $1$ in this paper. 
For actual programming, adjustments have to be made 
if arrays start from the index $0$ for a computer language.

In order to compute $\gB$, we have to consider 
the set of finite subsets of the set of weights of 
$V$ and find the closest point $\beta$ to the origin from the 
convex hull. 
%Even though the set of finite subsets of the set of weights of $V$ 
%is very large, there is a Weyl group action. This 
%makes it possible to reduce the number of cases 
%significantly.   We use C for this part. 
%One issue we have to worry about 
%when using C is that the size of integers do not exceed 
%certain limits. 
%Since the action of Weyl group basically 
%changes the order of numbers the size of integers 
%is not an issue in this step. 
%After reducing the number of cases, we compute 
%$\be$. This can be carried out by computing the canonical 
%form of a certain matrix.  We use MAPLE for this part, 
%since there is no restriction of digits in MAPLE. 

We first explain how to reduce the number of cases. 
Let $G,G_1,V$ be as in Section \ref{sec:git-stratification}. 
Let $r=\dim \gt^*$. 
We remind the reader that $\Gam=\{\gam_1\ccd \gam_N\}$ is the 
set of weights of coordinates of $V$. 
Let $A_R$ be the set consisting of all subsets of cardinality $R$
of $\Gam$.  If $C$ is a convex polytope 
then it is a finite union of simplices.  
Therefore, we only have to consider $\gI\in A_R$ 
which satisfy the following condition. 

\begin{cond}
\label{cond:be-condition}
\begin{itemize}
\item[(1)] 
$R\leq r$. 
\item[(2)] 
If $\gI=\{\gam_{j_1}\ccd \gam_{j_R}\}$ 
and $\be$ is the closest point of $\Conv \gI$ 
to the origin, then $\{\gam_{j_2}-\gam_{j_1}\ccd \gam_{j_R}-\gam_{j_1}\}$ 
is linearly independent and  
$\be$ is orthogonal to 
$\gam_{j_2}-\gam_{j_1}\ccd \gam_{j_R}-\gam_{j_1}$. 
\item[(3)] 
$\be$ is an interior point of $C$. 
\end{itemize}  
\end{cond}

We used the capital letter $R$ because this is the constant 
we shall use in algorithms. It will be easier this way to distinguish 
constants and variables in algorithms. 

Note that since $\Conv \gI$ does not contain the origin, 
we only have to consider the face of $\Conv \gI$ 
which contains $\be$. So we may assume that 
the dimension of $\Conv \gI$ is strictly less than $r$. 
Since an ($r-1$)-dimensional simplex is determined 
by $r$ vectors, the properties (1), (2) follow. 
The reason why we may assume (3) is that 
$\be$ can be obtained from $\gI'\in A_{R'}$ 
for $R'<R$ if $\be$ belongs to the boundary of $\Conv \gI$. 

Let $B_R$ be the set of all $\gI\in A_R$ which satisfies 
Condition \ref{cond:be-condition}. Obviously 
$\weyl$ acts on $B_R$.  Let $C_R\sub B_R$
be a set of representatives of $\weyl\backslash B_R$. 
Let $\gI\in C_R$ and $\be'$ be the closest point 
of $\Conv \gI$ to the origin. We choose an element 
$g\in \weyl$ so that $\be=g\be'\in \gt^*_+$. Let 
$S_R$ be the set of such $\be$.  

\begin{prop}
\label{prop:gB-computation}
$\gB = \bigcup_{R=1}^r S_R$. 
\end{prop}
\begin{proof}
It is enough to prove that 
$\gB \sub \bigcup_{R=1}^r S_R$. 
Suppose that $\be\in\gB$ is obtained from 
$\gI\in B_R$.  Then there exist 
$\gJ\in C_R$ and $g\in \weyl$ such that 
$\gJ=g \gI$.  Let $\be'$ be the closest point 
of $\Conv \gJ$ to the origin and $h\in\weyl$ 
is an element such that $h \be'\in \gt^*_+$. 

Since $\be$ is the closest point 
of $\Conv \gI$ to the origin, 
$g\beta=\be'$.  So $hg\beta\in \gt^*_+$, 
which implies that $h\be' = hg\beta=\beta$. 
Therefore, $\be\in S_R$. 
\end{proof}

By the above proposition, it is enough to 
determine $S_R$ for $R=1\ccd r$ and remove 
duplication. 

We explain the algorithm  more explicitly for the \pv s (1)--(4)
in the following. 
We choose products of $\spl$'s as $G_1$ in 
Section \ref{sec:git-stratification}. For example, 
$G_1=\spl_5\times \spl_4$ for the case (3).  
Let $T_0\sub G$ be the center of $G$.  For example, 
$T_0=\{(t_1I_6,t_2I_2)\mid t_1,t_2\in\gl_1\}$ for the case (2). 
Let $T\sub G_1$ be the subgroup consisting of elements whose components 
are diagonal matrices.
We choose $T$ in Section \ref{sec:git-stratification} for the cases (1)--(4) 
as follows.
\begin{align*}
(1) \;\; T = & 
\left\{(\diag(t_{11},t_{12},t_{13}),\diag(t_{21},t_{22},t_{23}),\diag(t_{31},t_{32}))
\,\vrule\,
\begin{matrix}
t_{11}t_{12}t_{13}=t_{21}t_{22}t_{23} \\
=t_{31}t_{32}=1
\end{matrix}
\right\}. \\
(2) \;\; T = & 
\{(\diag(t_{11}\ccd t_{16}),\diag(t_{21},t_{22}))
\mid t_{11}\cdots t_{16}=t_{21}t_{22}=1\}. \\
(3) \;\; T = & 
\{(\diag(t_{11}\ccd t_{15}),\diag(t_{21}\ccd t_{24}))
\mid t_{11}\cdots t_{15}=t_{21}\cdots t_{24}=1\}. \\
(4) \;\; T = & 
\{\diag(t_1\ccd t_8) \mid t_1\cdots t_8=1\}. 
\end{align*}

Then we can describe $\gt^*$ as follows.   
\begin{align*}
(1) \;\; \gt^* & = \left\{(a_{11},a_{12},a_{13},
a_{21},a_{22},a_{23},a_{31},a_{32})\in\R^8
\,\vrule\, \sum_{j=1}^3a_{1j}=\sum_{j=1}^3a_{2j}=\sum_{j=1}^2a_{3j}=0\right\}.  \\
(2) \;\; \gt^* & = \left\{(a_{11}\ccd a_{16},a_{21},a_{22})\in\R^8
\,\vrule\, \sum_{j=1}^6a_{1j}=\sum_{j=1}^2a_{2j}=0\right\}. \\
(3) \;\; \gt^* & = \left\{(a_{11}\ccd a_{15},a_{21}\ccd a_{24})\in\R^9
\,\vrule\, \sum_{j=1}^5a_{1j}=\sum_{j=1}^4a_{2j}=0\right\}. \\
(4) \;\; \gt^* & = \left\{(a_1\ccd a_8)\in\R^8
\,\vrule\, \sum_{j=1}^8a_j =0\right\}.
\end{align*}
For the case (3), $a=(a_{11}\ccd a_{15},a_{21}\ccd a_{24})\in \Z^8$ 
can be regarded as a character of $T$ so that for 
$t=(\diag(t_{11}\ccd t_{15}),\diag(t_{21}\ccd t_{24}))$, 
$t^a \stackrel{\rm{def}}= \prod_{i=1}^5 t_{1i}^{a_{1i}}\prod_{i=1}^4 t_{2i}^{a_{2i}}$. 
Other cases are similar. 
%Then $\gt^*_{\Q}$ consists of 
%$a=(a_{11}\ccd a_{15},a_{21}\ccd a_{24})\in \gt^*$ such that $a_{ij}\in\Q$ 
%for all $i,j$.  

The Weyl groups $\weyl$ for the cases (1)--(4) are
$\gS_3\times \gS_3\times \gS_2$, 
$\gS_6\times \gS_2$,
$\gS_5\times \gS_4$,
$\gS_8$ respectively. 
To define a $\weyl$-invariant inner product on $\gt$ is equivalent to 
define a $\weyl$-invariant inner product on $\gt^*$. 
For the case (3), we define 
\begin{equation*}
(a,b)_* = \sum_{i=1}^5 a_{1i}b_{1i} + \sum_{i=1}^4 a_{2i}b_{2i}
\end{equation*}
for $a=(a_{11}\ccd a_{15},a_{21}\ccd a_{24}), 
b=(b_{11}\ccd b_{15},b_{21}\ccd b_{24})$.
This inner product is $\weyl$-invariant. We define $(\;,\;)_*$ 
for other cases similarly.  We choose the Weyl chamber 
for the cases (1)--(4) as follows. 

\vskip 5pt
(1)
\begin{math}
\gt^*_+ = \left\{(a_{11},a_{12},a_{13},a_{21},a_{22},a_{23},a_{31},a_{32})
\in \gt^* \, \vrule \, 
\begin{matrix}
a_{i1}\leq a_{i2}\leq a_{i3} \;(i=1,2), \\
a_{31}\leq a_{32} 
\end{matrix}
\right\}.
\end{math}

(2)
\begin{math}
\gt^*_+ = \{(a_{11}\ccd a_{16},a_{21},a_{22})\in \gt^* \mid
a_{11}\leq \cdots \leq a_{16},a_{21}\leq a_{22} \}.
\end{math}

(3)
\begin{math}
\gt^*_+ = \{(a_{11}\ccd a_{15},a_{21}\ccd a_{24})\in \gt^*\mid 
a_{11}\leq \cdots \leq a_{15},a_{21}\leq \cdots\leq a_{24} \}.
\end{math}

(4)
\begin{math}
\gt^*_+ = \{(a_1\ccd a_8)\in \gt^*\mid 
a_1\leq \cdots \leq a_8\}.
\end{math}

\vskip 5pt

Let $\bbmp_{n,1}\ccd \bbmp_{n,n}$ be the coordinate vectors of $\aff^n$.
We put $p_{n,ij}=\bbmp_{n,i}\wedge \bbmp_{n,j}$, 
$p_{n,ijk}=\bbmp_{n,i}\wedge \bbmp_{n,j}\wedge \bbmp_{n,k}$, 
$q_{n,ij}=\bbmp_{n,i}\otimes \bbmp_{n,j}$. 
We choose a basis of $V$ for the cases (1)--(4) 
so that the coordinate vectors are as follows.
Let $N=\dim V$. 
Note that $N=18,30,40,56$ for the cases (1)--(4) respectively. 

\vskip 5pt
(1) 
\begin{math}
\mathbbm a_1= q_{3,11}\otimes \bbmp_{2,1},\;
\mathbbm a_2= q_{3,12}\otimes \bbmp_{2,1}\ccd 
\mathbbm a_9= q_{3,33}\otimes \bbmp_{2,1}\ccd 
\mathbbm a_{18}= q_{3,33}\otimes \bbmp_{2,2}.
\end{math}

(2) 
\begin{math}
\mathbbm a_1= p_{6,12}\otimes \bbmp_{2,1},\;
\mathbbm a_{15}= p_{6,56}\otimes \bbmp_{2,1},\;
\mathbbm a_{16}= p_{6,12}\otimes \bbmp_{2,2}\ccd 
\mathbbm a_{30}= p_{6,56}\otimes \bbmp_{2,2}. 
\end{math}

(3) 
\begin{math}
\mathbbm a_1= p_{5,12}\otimes \bbmp_{4,1},\;
\mathbbm a_{10}= p_{5,45}\otimes \bbmp_{4,1},\;
\mathbbm a_{11}= p_{5,12}\otimes \bbmp_{4,2}\ccd 
\mathbbm a_{40}= p_{5,45}\otimes \bbmp_{4,4}. 
\end{math}

(4) 
\begin{math}
\mathbbm a_1= p_{8,123},\;
\mathbbm a_2= p_{8,124} \ccd 
\mathbbm a_{56}= p_{8,678}
\end{math}

\vskip 5pt
Let $\gam_i$ be the weight of $\mathbbm a_i$. 
Then $\{\gam_1\ccd \gam_N\}$ for the cases (1)--(4) are as follows.

\vskip 10pt

(1)
\begin{math}
\gam_1 = (\tfrac 23,-\tfrac 13,-\tfrac 13,\tfrac 23,-\tfrac 13,-\tfrac 13,
\tfrac 12,-\tfrac 12)\ccd 
\gam_{18} = (-\tfrac 13,-\tfrac 13,\tfrac 23,-\tfrac 13,-\tfrac 13,\tfrac 23,
-\tfrac 12,\tfrac 12).
\end{math}

(2)
\begin{math}
\gam_1 = (\tfrac 23,\tfrac 23,-\tfrac 13,-\tfrac 13,-\tfrac 13,-\tfrac 13,
\tfrac 12,-\tfrac 12)\ccd 
\gam_{30} = (-\tfrac 13,-\tfrac 13,-\tfrac 13,-\tfrac 13,\tfrac 23,\tfrac 23,
-\tfrac 12,\tfrac 12). 
\end{math}

(3)
\begin{math}
\gam_1 = \left(\tfrac 35,\tfrac 35,-\tfrac 25,-\tfrac 25,-\tfrac 25,
\tfrac 34,-\tfrac 14,-\tfrac 14,-\tfrac 14\right)
\end{math}

\quad\quad 
\begin{math}
\ccd \gam_{40} = \left(-\tfrac 25,-\tfrac 25,-\tfrac 25,\tfrac 35,\tfrac 35,
-\tfrac 14,-\tfrac 14,-\tfrac 14,\tfrac 34\right).
\end{math}

(4)
\begin{math}
\gam_1 = (\tfrac 58,\tfrac 58,\tfrac 58,-\tfrac 38\ccd -\tfrac 38) \ccd 
\gam_{56} = (-\tfrac 38\ccd -\tfrac 38,\tfrac 58,\tfrac 58\ccd \tfrac 58).  
\end{math}

\vskip 5pt
We fix $1\leq R\leq r$ 
($r=5,6,7,7$ for the cases (1)--(4) respectively). 
Let $A_R$ be the set of all subsets $Y$ of $\Gam$ such that 
$\# Y = R$. We identify $A_R$ with the set of all sequences 
$v=(v_1\ccd v_R)$ such that $1\leq v_1<\cdots<v_R\leq N$. 

\vskip 5pt
\noindent
Step 1. 
We find a set $C_R'$ of representatives of $\weyl \backslash A_R$.
For this purpose, we assign the lexicographical order 
to any element $\gI\in A_R$, say $L(\gI)$. 
Note that 
\begin{math}
1\leq L(\gI)\leq 
\left(
\begin{smallmatrix}
N \\ R
\end{smallmatrix}
\right).
\end{math}
For 
\begin{math}
1\leq i\leq 
\left(
\begin{smallmatrix}
N \\ R
\end{smallmatrix}
\right), 
\end{math}
let $\gI(i)\in A_R$ be the subset such that 
$L(\gI(i))=i$.

Let $X_{R,i}$ be an array of integers 
such that 
\begin{math}
1\leq i\leq \left(
\begin{smallmatrix}
N \\ R
\end{smallmatrix}
\right).
\end{math}
At first we assign $X_{R,i}=0$ for all $i$. 
We change the value of $X_{R,1}$ to $1$ and then 
for all $w\in \weyl$, 
change the value of $X_{R,L(w(\gI(1)))}$ to $2$. 
Then we consider the first $i\geq 2$ such that $X_{R,i}=0$,
change the value of $X_{R,i}$ to $1$ and for all $w\in \weyl$, 
change the value of $X_{R,L(w(\gI(i)))}$ to $2$. 
We continue this process.  Then $C_R'=\{\gI(i)\mid X_{R,i}=1\}$ 
is a set of representatives for $\weyl \backslash A_R$.

\vskip 5pt
\noindent
Step 2. 
Suppose that $\gI=\{\gam_{j_1}\ccd \gam_{j_R}\}\in C_R'$. 
We would like to find the closest point $\be'$ of $\Conv \gI$ to the 
origin and see if it satisfies Condition \ref{cond:be-condition}. 
Such $\be'$ is in the form 
$\be'=c_1\gam_{j_1}+\cdots+c_i\gam_{j_R}$
where $c_1\ccd c_R\in\Q$ and 
$0<c_1\ccd c_R<1$, $c_1+\cdots+c_R=1$.

Let $M=(m_{kl})$ be the $R\times R$ matrix such that
$m_{kl}=(\gam_{j_k}-\gam_{j_1},\gam_{j_l})_*$ for $k=2\ccd R,l=1\ccd R$
and $m_{11}=\cdots=m_{1R}=1$. We put $c=[c_1\ccd c_R]$. 
Then $\be'$ is orthogonal to 
$\gam_{j_2}-\gam_{j_1}\ccd \gam_{j_R}-\gam_{j_1}$
if and only if entries of $Mc$ are $0$ except for the first entry. 
The condition $c_1+\cdots+c_R=1$ means that the first entry 
of $Mc$ is $1$. So $c$ has to satisfy the condition 
$Mc=[1,0\ccd 0]$. 

If $\{\gam_{j_2}-\gam_{j_1}\ccd \gam_{j_R}-\gam_{j_1}\}$
is linearly independent then $c$ is unique and so $M$ has to be
non-singular. We put $b=[1,0\ccd 0]$. 
We form the augmented matrix $M'=(M\; b)$ and 
find the reduced row echelon form of $M'$, say $(M_0\; d)$ 
($d=[d_1\ccd d_R]$).
Then $M$ is non-singular if and only if 
the $(R,R)$-entry of $M_0$ is $1$. 
Let $C_R$ be the set of $\gI\in C_R'$ 
such that the $(R,R)$-entry of $M_0$ is $1$ 
and $d_1\ccd d_R>0$. Then this $C_R$ can be regarded as 
$C_R$ described before Proposition \ref{prop:gB-computation}. 
If $\gI\in C_R$ then we form
$\be'=d_1\gam_{j_1}+\cdots+d_R\gam_{j_R}$. 
By sorting entries of $\be'$, we obtain an element $\be\in \gt^*_+$. 

\vskip 5pt
\noindent
Step 3. We combine all $\be$ obtained in Step 2 for $R=1\ccd r$. 
We remove duplication and the zero vector.  Then the 
list obtained is the set $\gB$.

\section{lexicographical order of combinations}
\label{sec:lexic-order-comb}

Let $G,G_1,V,N,\Gam,r$ be as in Section \ref{sec:git-stratification}. 

To assign a multiple array such as $a_{j_1\ccd j_R}$ 
to the subset $\gI= \{\gam_{j_1}\ccd \gam_{j_R}\}$ 
is not a good idea, because it consumes unnecessary 
memory space. So we consider the lexicographical order to $\gI$
and assign a single array. 
To use this idea, we must have a way to go back and forth
between such combinations and their lexicographical orders. 
The purpose of this section is to explain this correspondence
explicitly. 

%As we stated in Section \ref{sec:outline-program}, 
%we assign the lexicographical order to a subset of order $R$ of $\Gam$
%where $1\leq R\leq r$. 

Let 
\begin{math}
\left( 
\begin{smallmatrix}
n \\ m 
\end{smallmatrix}
\right)
\end{math}
be the binomial coefficient. We consider integers 
$n\geq m\geq 0$ except for  $m=n+1$, 
where we define  
\begin{math}
\left(
\begin{smallmatrix}
n \\ n+1
\end{smallmatrix}
\right)
=0.
\end{math}
Note that  
\begin{math}
\left(
\begin{smallmatrix}
n \\ n
\end{smallmatrix}
\right)
= \left(
\begin{smallmatrix}
n \\ 0
\end{smallmatrix}
\right)=1.
\end{math}

Let $N\geq R\geq 1$ be integers. 
Let $A(N,R)$ be the set of sequences 
$c=(c_1\ccd c_R)$ of integers such that
$1\leq c_1<c_2<\cdots<c_R\leq N$ (this is $A_R$ in the previous section). 
For such $c$, let $L(N,R,c)\geq 1$ be its lexicographical order. 
For example, if $N=3,R=2,c_1=2,c_2=3$ 
then $L(3,2,c)=3$. 

We would like to express $L(N,R,c)$ in terms of $c$ and vice versa. 

\begin{prop}
\label{prop:comb2number}
If $N\geq R\geq 1$ and $c\in A(N,R)$ then 
\begin{equation}
\label{eq:Nnrc-formula}
L(N,R,c) = \sum_{i=1}^R
\left( 
\begin{pmatrix}
N-i \\ R-i+1 
\end{pmatrix}
-\begin{pmatrix}
N-c_i \\
R-i+1
\end{pmatrix}
\right)+1
\end{equation}
\end{prop}
\begin{proof}
Note that $c_i\leq N-R+i$. So 
$N-c_i\geq R-i$ and 
$N-c_i= R-i$ if and only if 
$c_i=N-R+i$. If 
$c_i=N-R+i$ then 
\begin{math}
\left(
\begin{smallmatrix}
N-c_i\\
R-i+1 
\end{smallmatrix}
\right)
= \left(
\begin{smallmatrix}
R-i \\ R-i+1
\end{smallmatrix}
\right)=0. 
\end{math}

If $R=1$ then $L(N,R,c)=c_1$ and 
(\ref{eq:Nnrc-formula}) is valid in this case. 

Suppose that $R>1$. 
We put $d_1=c_2-c_1\ccd d_{R-1}=c_R-c_1$. 
Then $d\in A(N-c_1,R-1)$. 
If $c_1=1$ then $L(N,R,c)=L(N-1,R-1,d)$. 
If $c_1>1$ then the number of 
$c'=(c'_1\ccd c'_R)\in A(N,R)$ such that 
$c'_1<c_1$ is 
\begin{equation*}
\begin{pmatrix}
N \\ R 
\end{pmatrix}
- \begin{pmatrix}
N- c_1+1 \\ R
\end{pmatrix}. 
\end{equation*}
Note that the number of $m$ such that $c_1\leq m\leq N$ 
is $N-c_1+1$. So 
\begin{equation}
\label{eq:L(n,r,i)-recursive}
L(N,R,c) = 
\begin{pmatrix}
N \\ R 
\end{pmatrix}
- \begin{pmatrix}
N- c_1+1 \\ R
\end{pmatrix}
+ L(N-c_1,R-1,d). 
\end{equation}
This formula is valid in the case $c_1=1$ also. 

This formula implies by induction that 
\begin{equation}
\label{eq:L(n,r,i)-middle}
L(N,R,c) = 
\begin{pmatrix}
N \\ R 
\end{pmatrix}
-\begin{pmatrix}
N-c_1+1 \\ R 
\end{pmatrix}
+ \sum_{i=1}^{R-1} 
\left(
\begin{pmatrix}
N-c_1-i \\
R-i
\end{pmatrix}
-\begin{pmatrix}
N-c_{i+1} \\
R-i
\end{pmatrix}
\right)+1.
\end{equation}
This formula is valid for the case $R=1$ also. 

Note that if $0\leq m\leq N$ then 
\begin{equation*}
\begin{pmatrix}
N \\ m 
\end{pmatrix}
= \sum_{i=0}^{m-1}
\begin{pmatrix}
N-i-1 \\ m-i 
\end{pmatrix}+1.
\end{equation*}
So 
\begin{equation*}
\begin{pmatrix}
N \\ R 
\end{pmatrix}
-\begin{pmatrix}
N-c_1+1 \\ R 
\end{pmatrix}
= \sum_{i=0}^{R-1} 
\left(
\begin{pmatrix}
N-i-1 \\ R-i 
\end{pmatrix}
-\begin{pmatrix}
N-c_1-i \\ R-i 
\end{pmatrix}
\right).
\end{equation*}
Since 
\begin{align*}
& \sum_{i=0}^{R-1} 
\begin{pmatrix}
N-i-1 \\ R-i 
\end{pmatrix}
= \sum_{i=1}^R 
\begin{pmatrix}
N-i \\ R-i+1 
\end{pmatrix}
\end{align*}
and
\begin{align*}
& -\sum_{i=0}^{R-1} 
\begin{pmatrix}
N-c_1-i \\ R-i 
\end{pmatrix}
+ \sum_{i=1}^{R-1} 
\left(
\begin{pmatrix}
N-c_1-i \\
R-i
\end{pmatrix}
-\begin{pmatrix}
N-c_{i+1} \\
R-i
\end{pmatrix}
\right)
= -\sum_{i=1}^R
\begin{pmatrix}
N-c_i \\
R-i+1
\end{pmatrix},
\end{align*}
we obtain (\ref{eq:Nnrc-formula}). 
\end{proof}

We consider the opposite direction. 
Let $m>0$ be an integer such that 
\begin{math}
m\leq \left(
\begin{smallmatrix}
N \\ R 
\end{smallmatrix}
\right).
\end{math}
We would like to find $c\in A(N,R)$ such that 
$L(N,R,c)=m$. 

We put $m_1=m$. 
For $i=2\ccd R$, we put  
\begin{equation*}
m_i = 
m - \sum_{l=1}^{i-1} 
\left(
\begin{pmatrix}
N-l \\ R-l+1
\end{pmatrix}
- \begin{pmatrix}
N-c_l \\ R-l+1
\end{pmatrix}
\right). 
\end{equation*}
Note that $m_i$ does not depend on $c_i\ccd c_R$. 

\begin{prop}
\label{prop:c-characterize}
If  $L(N,R,c)=m$ then 
$c_i$ is characterized by the following condition:
\begin{equation*}
\begin{pmatrix}
N-i+1 \\ R-i+1
\end{pmatrix}
- \begin{pmatrix}
N-c_i+1 \\ R-i+1
\end{pmatrix}
< m_i \leq 
\begin{pmatrix}
N-i+1 \\ R-i+1
\end{pmatrix}
- \begin{pmatrix}
N-c_i \\ R-i+1
\end{pmatrix}. 
\end{equation*}
\end{prop}
\begin{proof}
By the consideration of (\ref{eq:L(n,r,i)-recursive}), 
$c_1\geq j$ if and only if 
\begin{math}
m> \left(
\begin{smallmatrix}
N \\ R  
\end{smallmatrix}
\right)
-\left(
\begin{smallmatrix}
N-j+1 \\ R 
\end{smallmatrix}
\right). 
\end{math}
Therefore, $c_1$ is characterize by the following formula:
\begin{equation*}
\begin{pmatrix}
N \\ R
\end{pmatrix}
-\begin{pmatrix}
N-c_1+1 \\ R
\end{pmatrix}
< m \leq 
\begin{pmatrix}
N \\ R
\end{pmatrix}
-\begin{pmatrix}
N-c_1 \\ R 
\end{pmatrix}. 
\end{equation*}
So the statement of the proposition holds for $i=1$.

Let $d_1=c_2-c_1\ccd d_{R-1}=c_R-c_1$.
Then $d\in A(N-c_1,R-1)$. 
We put $m_2' = L(N-c_1,R-1,d)$. By (\ref{eq:L(n,r,i)-recursive}), 
\begin{equation*}
m_2' = m - \left(
\begin{pmatrix}
N \\ R
\end{pmatrix}
-\begin{pmatrix}
N-c_1+1 \\ R
\end{pmatrix}
\right)
\end{equation*}
Since $c_1+d_1=c_2$, $c_2$ is characterized 
by the following formula:
\begin{equation*}
\begin{pmatrix}
N-c_1 \\ R-1 
\end{pmatrix}
-\begin{pmatrix}
N-c_2+1 \\ R-1 
\end{pmatrix}
< m_2' \leq
\begin{pmatrix}
N-c_1 \\ R-1 
\end{pmatrix}
-\begin{pmatrix}
N-c_2 \\ R-1 
\end{pmatrix}. 
\end{equation*}

By continuing this process, for $i=2\ccd R$, 
$c_i$ is characterized by the following 
condition: 
\begin{equation}
\label{eq:ij-condition}
\begin{pmatrix}
N-c_{i-1} \\ R-i+1 
\end{pmatrix}
-\begin{pmatrix}
N-c_i+1 \\ R-i+1 
\end{pmatrix}
< m_i' \leq
\begin{pmatrix}
N-c_{i-1} \\ R-i+1 
\end{pmatrix}
-\begin{pmatrix}
N-c_i \\ R-i+1 
\end{pmatrix}
\end{equation}
where for $i=2\ccd R$, 
\begin{equation*}
m_i' = m - \left(\begin{pmatrix}
N \\ R 
\end{pmatrix}
-\begin{pmatrix}
N-c_1+1 \\ R
\end{pmatrix}
\right)
-\sum_{l=1}^{i-2}
\left(
\begin{pmatrix}
N-c_l \\ R-l 
\end{pmatrix}
-\begin{pmatrix}
N-c_{l+1}+1 \\ R-l 
\end{pmatrix}
\right). 
\end{equation*}

We put 
\begin{math}
m_i''= m_i' 
+ \begin{pmatrix}
N-i+1 \\ R-i+1 
\end{pmatrix}
-\begin{pmatrix}
N-c_{i-1} \\ R-i+1
\end{pmatrix}.
\end{math}
Since
\begin{equation*}
\begin{pmatrix}
N \\ R 
\end{pmatrix} 
- \begin{pmatrix}
N-i+1 \\ R-i+1
\end{pmatrix}
= \sum_{l=1}^{i-1} 
\begin{pmatrix}
N-l \\ R-l+1
\end{pmatrix}
\end{equation*}
and
\begin{align*}
& \begin{pmatrix}
N-c_1+1 \\ R 
\end{pmatrix}
- \sum_{l=1}^{i-2}
\left(
\begin{pmatrix}
N-c_l \\ R-l 
\end{pmatrix}
-\begin{pmatrix}
N-c_{l+1}+1 \\ R-l 
\end{pmatrix}
\right) 
- \begin{pmatrix}
N-c_{i-1} \\ R-i+1
\end{pmatrix} \\
& = \begin{pmatrix}
N-c_1 \\ R
\end{pmatrix}
+ \begin{pmatrix}
N-c_2 \\ R-1
\end{pmatrix}
+ \cdots + 
\begin{pmatrix}
N-c_{i-1} \\ R-i+2 
\end{pmatrix}
= \sum_{l=1}^{i-1} 
\begin{pmatrix}
N-c_l \\ R-l+1
\end{pmatrix},
\end{align*}
we have 
\begin{equation*}
m_i'' = m - \sum_{l=1}^{i-1} 
\left(
\begin{pmatrix}
N-l \\ R-l+1
\end{pmatrix}
- \begin{pmatrix}
N-c_l \\ R-l+1
\end{pmatrix}
\right) = m_i.
\end{equation*}
This implies that the condition (\ref{eq:ij-condition})
is equivalent to the condition in the statement of 
this proposition. 
\end{proof}

Let 
\begin{equation}
\label{eq:a_ij}
a_{i,j} = 
\begin{pmatrix}
N-i \\ R-i+1 
\end{pmatrix}
- \begin{pmatrix}
N-j \\ R-i+1
\end{pmatrix}
\end{equation}
for $i=1\ccd R$, $j=i\ccd N-R+i$ and 
\begin{equation}
\label{eq:b_ij}
b_{i,j} = 
\begin{pmatrix}
N-i+1 \\ R-i+1 
\end{pmatrix}
- \begin{pmatrix}
N-j \\ R-i+1
\end{pmatrix}
\end{equation}
for $i=1\ccd R$, $j=i\ccd N-R+i$. 

Proposition \ref{prop:comb2number} implies that  
if $c\in A(N,R)$ and $m=L(N,R,c)$ then 
\begin{math}
m = \sum_{i=1}^{R} a_{i,c_i}. 
\end{math}
Also Proposition \ref{prop:c-characterize} implies that 
if $m_1=m$, $m_i = m-\sum_{l=1}^{i-1} a_{l,c_l}$ ($i=2\ccd R$) 
then $c_i$ is the smallest integer $i\leq j\leq N-R+i$ such that 
$b_{i,j}\geq m_i$.

\section{Algorithms}
\label{sec:programs}

In this section, we describe some details of algorithms 
to find the set $\gB$ for the \pv s (1)--(4). 
We describe the algorithms so that they do not 
depend on particular computer languages here. 
As we stated in Section \ref{sec:outline-program}, 
we assume that arrays start from the index $1$ 
even though arrays start from the index $0$ 
in some computer languages. 
The actual computer programs are made public 
in the second author's home page 
(https://www.math.kyoto-u.ac.jp/\~{}yukie/Strata-pub.zip). 
The letters we use in algorithms are different
from those used in actual programs, since 
in actual programs, variables like \verb|i1|, \verb|i2| 
are used and it may be confusing to use such names 
to explain the algorithms.

We use the formulation of Sections \ref{sec:git-stratification},
\ref{sec:outline-program}. 
For the \pv s (1)--(4), let $G_1,T,T_0,T_1,\gt^*,\gt^*_+,\weyl$ be as in 
Section \ref{sec:outline-program}. 

We consider Steps 1--3 of Section \ref{sec:outline-program}.

\subsection{Step 1}

Let $B_i$ be the set in Section \ref{sec:outline-program}
(see the paragraph above Proposition \ref{prop:gB-computation}). 
We find a set of representatives for $\weyl\backslash B_i$ 
in this step. 

It is fairly easy to generate permutations. 
We assume that elements of $\gS_2\ccd \gS_8$ 
have been generated and stored in a file
as $P^{(2)}_{ij}\ccd P^{(8)}_{ij}$. For example, 
\begin{equation*}
P^{(3)}_{11}=1,P^{(3)}_{12}=2,P^{(3)}_{13}=3\ccd 
P^{(3)}_{61}=3,P^{(3)}_{62}=2,P^{(3)}_{63}=1. 
\end{equation*}

Let $N=18,30,40,56$ for the \pv s (1)--(4) respectively. 
Let $\mathbbm a_1\ccd \mathbbm a_N$ be the coordinate 
vectors defined in Section \ref{sec:outline-program}
and $\gam_1\ccd \gam_N\in\gt^*$ their weights. 
To consider a subset $\gI\sub \{\gam_1\ccd \gam_N\}$
such that $\# \gI=R$ is the same 
as to consider combinations of $R$ numbers
from $N$ numbers.  If $\gI=\{\gam_{i_1}\ccd \gam_{i_R}\}$
where $i_1<\cdots<i_R$ then we assign the lexicographical order
of $\{i_1\ccd i_R\}$ to $\gI$. 
As we stated in Introduction, we use the lexicographical order 
to keep track of $\gI$.

For each $R$, we carry out algorithms in Steps 1,2. 
So algorithms in these steps depend on $R$. 
In Step 3, we combine results of Steps 1,2 for 
all $R$, remove duplication and obtain necessary informations
for each $\be\in\gB$.

If $v=(v_1\ccd v_R)$ is an array of distinct integers then 
the algorithm to sort $v_1\ccd v_R$ is well-known 
and we leave the details to the reader. It returns
an array $w_1<\dots< w_R$ obtained from changing the
order of $v_1\ccd v_R$.  Also, it is easy to 
compute binomial coefficients 
\begin{math}
\left(
\begin{smallmatrix}
n \\ m
\end{smallmatrix}
\right)
\end{math}
by Pascal's identity and we will not describe the details.

For an array of distinct $R$ integers $v=(v_1\ccd v_R)$ such that 
$1\leq v_1\ccd v_R\leq N$, let ${\bf CombN}(v)$ be the 
function which sorts $v_1\ccd v_R$ so that 
$v_1<\cdots <v_R$ and 
returns the lexicographical order of $v$.
For
\begin{math}
1\leq m\leq 
\left(
\begin{smallmatrix}
N \\ R
\end{smallmatrix}
\right), 
\end{math}
let ${\bf NComb}(m,v)$ be the function which makes $v$ 
the sequence $(v_1\ccd v_R)$ such that 
$1\leq v_1<\cdots<v_R\leq N$ and that ${\bf CombN}(v)=m$. 
When we use these functions, we assume that 
the values of $N,R$ are set. 
These functions $\text{\bf CombN}(v)$, $\text{\bf NComb}(m,v)$  
can be computed by Propositions \ref{prop:comb2number}, 
\ref{prop:c-characterize} as follows. 
The values of $a_{i,j}$ and $b_{i,j}$ in (\ref{eq:a_ij}), 
(\ref{eq:b_ij}) are heavily used. 
So the following values should be computed before other algorithms. 

\begin{itemize}
\item[(1)]
$a_{i,j}$ in (\ref{eq:a_ij}) for $i=1\ccd R$, $j=i\ccd N-R+i$. 
\item[(2)]
$b_{i,j}$ in (\ref{eq:b_ij}) for $i=1\ccd R$, $j=i\ccd N-R+i$.
\end{itemize}

\begin{algo}
\label{algo:programs-NComb}
\upshape
$\;$

\vskip 5pt
\noindent
(i) {Name:} ${\bf CombN}(v)$ 

\hskip 5pt {Require:} $v=(v_1\ccd v_R)$: an array of 
elements of $\bbn$.  

\hskip 5pt
{Description:} 
It returns the lexicographical order of $v$ after sorting 
as ${\bf CombN}(v)$.  

\hskip 5pt
{Local variables:} $i\in {\bbn}$.  

\vskip 10pt
\noindent
1. Sort $v$ so that $v_1<\cdots<v_R$. 

\vskip 5pt
\noindent
2. Return the value $\sum_{i=1}^Ra_{i,v_i}$ as ${\bf CombN}(v)$. 
 
\vskip 5pt
\noindent
(ii) {Name:} ${\bf NComb}(m,v)$

\hskip 7pt {Require:}
$m\in \bbn$, $v=(v_1\ccd v_R)$: an array of elements of $\bbn$. 

\hskip 7pt
{Description:} It makes $v$ a sorted array whose 
lexicographical order is $m$. 

\hskip 7pt
{Local variables:} $i,j\in {\bbn}$, 
$l=(l_1\ccd l_R)$: an array of $R$ elements of $\bbn$.

\vskip 5pt
\noindent
1. $l_1\leftarrow m$ and $j\leftarrow 1$. 

\vskip 5pt
\noindent
2. If $b_{1,j}< l_1$ then $j \leftarrow j+1$ and repeat. 

\vskip 5pt
\noindent
3. $v_1\leftarrow j$. 

\vskip 5pt
\noindent
4. For $i=2\ccd R$ do the following. 

\vskip 5pt
4-a. $l_i \leftarrow m-\sum_{l=1}^{i-1} a_{l,v_l}$  and $j\leftarrow i$.

\vskip 5pt
4-b. If $b_{i,j}< l_i$ then $j \leftarrow j+1$ and repeat. 

\vskip 5pt
4-c. $v_i\leftarrow j$. 

\vskip 5pt
This finishes the algorithm. \hfill $\diamond$
\end{algo}

Note that ${\bf NComb}(m,v)$ is a ``void type'' function with 
no returned value. 

Now we consider the \pv s (1)--(4). 
We can make algorithms so that they are common 
for the cases (1)--(4) except for definitions 
of some constants and some subroutines.  
So we basically explain algorithms for the case (3). 
In the following, $(G,V)$ is the \pv{} (3).  
We consider Step 1 of Section \ref{sec:outline-program}.

We first have to describe the action of $\weyl\cong \gS_5\times \gS_4$ 
on $\{\gam_1\ccd \gam_{40}\}$. 
The order of $\weyl$ is $2880$. Each element of $\weyl$ 
induces an element of $\gS_{40}$. So to describe 
the action of $\weyl$ on $\{\gam_1\ccd \gam_{40}\}$, 
it is enough to assign an array of $40$ integers.

Elements of $\weyl$ are pairs $(t_1,t_2)$ of permutations 
$t_1\in\gS_5,t_2\in\gS_4$. They are arrays of $5,4$ integers.
Let $t_1(i)$ ($i=1\ccd 5$), $t_2(j)$ ($j=1\ccd 4$) be the values of 
$t_1,t_2$. Since the coordinate system of $V$ involves $\wedge^2 \aff^5$,
we have to consider combinations of $2$ elements of $\{1\ccd 5\}$. 
So even though we set $N:=40,R:=1\ccd 7$ in the main algorithm 
of Step 1, to describe the action of $\weyl$ on 
$\{\gam_1\ccd \gam_{40}\}$, we set $N:=5,R:=2$ to use 
the functions $\text{\bf CombN}(v),\text{\bf NComb}(m,v)$.

We define some constants as follows. 
\begin{equation}
\label{eq:constants-1}
\begin{aligned}
& N:= 5,\; R:= 2,\;   A:= 5,\;   B:= 4,\; C:= 40, \\
& L_1:= 120,\;   L_2:= 24,\;   L:= 2880,\;   N_1:= 10.	
\end{aligned}
\end{equation}

We consider the lexicographical order of 
combinations of $2$ numbers from $\{1\ccd 5\}$. 
We order coordinates of $V$ by associating the order 
$1\ccd 10$ (resp. $11\ccd 20$, etc.,) to coordinates 
whose second tensor factor 
is $[1,0,0,0]$ (resp. $[0,1,0,0]$, etc.,).  
Let $i=1\ccd 10,j=1\ccd 4$ and
$v=(v_1,v_2)$ be the $i$-th combination.  
Then by $(t_1,t_2)$, the ($N_1(j-1)+i$)-th coordinate 
is mapped to the ($N_1(t_2(j)-1)+k$)-th coordinate 
where $k$ is the lexicographical order of 
the combination $(t_1(v_1),t_1(v_2))$ (after sorted).

\begin{algo}
\label{algo:programs-weyl-action54}
\upshape

{Name:} ${\bf makeweyl54}(t_1,t_2,t_3)$ 

\vskip 5pt
{Require:} $t_1\in\gS_A,t_2\in\gS_B,t_3\in\gS_C$. 

{Description:} 
It makes $t_3$ the result of the action of $(t_1,t_2)\in\weyl$ 
on $\{\gam_1\ccd \gam_C\}$. 

{Local variables:} $i,j,k\in\bbn$, 
$v=(v_1,v_2),w=(w_1,w_2)$: arrays of elements of $\bbn$.

\vskip 10pt
\noindent
1. For $i=1\ccd N_1$, $j=1\ccd B$, do the following. 

\vskip 5pt
1-a. ${\bf NComb}(i,v)$.

\vskip 5pt
1-b. $w_1\leftarrow t_1(v_1),\; w_2\leftarrow t_1(v_2)$.

\vskip 5pt
1-c. $k \leftarrow {\bf CombN}(w)$, $t_3(N_1(j-1)+i)\leftarrow N_1(t_2(j)-1)+k$. 

\vskip 5pt
This is the end of the function
${\bf makeweyl54}(t_1,t_2,t_3)$. \hfill $\diamond$
\end{algo}

\begin{algo}
\label{algo:programs-weyl-list}
\upshape
{Description:} This algorithm lists the action of all elements of 
$\weyl=\gS_A\times \gS_B$ on $\{\gam_1\ccd \gam_C\}$. Since it is not a function, 
it does not require any variable as an input. However, 
constants in (\ref{eq:constants-1})
have to be defined and elements of $\gS_A,\gS_B$ 
have to be read from a file as 
$P^{(5)}_{ij}$ ($i=1\ccd L_1,j=1\ccd A$), 
$P^{(4)}_{ij}$ ($i=1\ccd L_2,j=1\ccd B$).

\vskip 5pt
{Local variables:}
$i,j,k\in\bbn$,  $W_k\in \gS_C$ ($k=1\ccd L$). 

\vskip 10pt
\noindent
1. Initialize $k\leftarrow 1$. 

\vskip 5pt
\noindent
2. For $i=1\ccd L_1$ and $j=1\ccd L_2$, do the following.

\vskip 5pt
2-a. ${\bf makeweyl54}(P^{(5)}_{i*},P^{(4)}_{j*},W_k)$.

\vskip 5pt 
2-b. $k\leftarrow k+1$ (every time 2-a is done for a pair $(i,j)$). 

\vskip 5pt 

Note that for fixed $i$, 
$P^{(5)}_{i*}\in \gS_A$ ($*=1\ccd A$).  
We regard $P^{(4)}_{j*}\in \gS_B$ similarly.  
 
\vskip 5pt
\noindent
3. Record $W_k$ ($k=1\ccd L$) in a file.  

\vskip 5pt
This finishes the algorithm. 
\hfill $\diamond$
\end{algo}

After this algorithm, we set $N:=40,R:=1\ccd 7$. 	
Now we consider the algorithm to reduce the number of cases.

We have to consider simplices of dimensions $0\ccd 6$. 
Since the Weyl group acts transitively on the set of coordinates, 
$B_1$ (see Proposition \ref{prop:gB-computation}) is a 
single $\weyl$-orbit.  The weight of the last coordinate 
is in the Weyl chamber and so $S_1$ consists of the weight 
of the last coordinate.   

Let $R$ be the number of 
vectors which determine an ($R-1$)-dimensional 
simplex.  We consider the cases $R=2\ccd 7$. 

\begin{algo}
\label{algo:step1-main}
\upshape
{\bf (The main algorithm for Step 1)}\,
{Description:} This algorithm determines $S_7$
of Proposition \ref{prop:gB-computation}. 
Constants $A,B,C,L_1,L_2,L$ have to be defined as in (\ref{eq:constants-1}).
Define
\begin{math}
R:= 7,M:= \left(
\begin{smallmatrix}
40 \\ 7
\end{smallmatrix}
\right) = 18643560. 
\end{math}
We set the environment so that we can use the functions 
${\bf CombN},{\bf NComb}$ for $N:=40,R:=7$. 
The list of $W_k\in\gS_C$ ($k=1\ccd L=2880$)
have to be read from a file. 

\vskip 5pt
{Local variables:} (i) $i,j,k,l,m\in\bbn$.

\quad
(ii) $t_1=(t_1(1)\ccd t_1(R))$, 
$t_2=(t_2(1)\ccd t_2(R))$: arrays of elements of $\bbn$. 

\quad
(iii) $X=(X_1\ccd X_M)$: an array of elements of $\bbn$. 
($X_i$ is $X_{R,i}$ of Section \ref{sec:outline-program}.) 

\quad
(iv) $v=(v_{i,j})$: an $M\times R$ matrix 
with entries in $\bbn$.

\vskip 5pt
\noindent
1. Initialize $X_i \leftarrow 0$ for $i=1\ccd M$. 

\vskip 5pt
\noindent
2. Initialize $m\leftarrow 0$.

\vskip 5pt
\noindent
3. For $i=1\ccd M$, if $X_i=0$, 
do the following (if $X_i\not=0$ then do nothing).  

\vskip 5pt
3-a. $m\leftarrow m+1$.

\vskip 5pt
3-b. For $j=1\ccd R$, $v_{m,j}\leftarrow t_1(i)$. 

\vskip 5pt
3-c.  $X_i\leftarrow 1$.

\vskip 5pt
3-d. ${\bf NComb}(i,t_1)$.

\vskip 5pt
3-e. For $j=1\ccd L$, do the following. 

\vskip 5pt
\quad \quad 
3-e-1. For $k=1\ccd R$, $t_2(k) \leftarrow W_j(t_1(k))$.  

\vskip 5pt
\quad \quad  
3-e-2. $l={\bf CombN}(t_2)$ and if $l>i$, $X_l \leftarrow 2$.  

%\vskip 5pt
%3-f. $i\leftarrow i+1$ and go back to 3-a. 

\vskip 5pt
\noindent
4. Record $v$.  

\vskip 5pt
This finishes the algorithm. 
\hfill $\diamond$
\end{algo}

\vskip 5pt
Note that in the step 3-e-1, $t_2$ is the result of the action of 
the Weyl group element $W_j$ to $t_1$.
Even though the size of $v$ is $M\times R$, 
$v_i$ is recorded only for $i$ from $1$ to the final value of $m$. 
It turns out that the final value of $m$ is $7891$ in this case.

For $R:=6\ccd 2$, we simply change the value of 
$M$ to 
\begin{math}
M := \left(
\begin{smallmatrix}
N \\ R
\end{smallmatrix}
\right)
\end{math}
and Algorithm \ref{algo:step1-main} works.

For the \pv{} (2), we can record the action of $\weyl$ in the same manner as 
in Algorithms \ref{algo:programs-weyl-action54}, \ref{algo:programs-weyl-list}
after changing the constants as follows (but one has to use 
$P^{(6)}_{ij},P^{(2)}_{ij}$ in Algorithm \ref{algo:programs-weyl-list}). 

\vskip 10pt

\begin{center}
 
\begin{tabular}{|c|c|c|c|c|c|c|c|c|}
\hline
$N$ & $R$ & $A$ & $B$ & $C$ & $L_1$ 
& $L_2$ & $L$ & $N_1$  \\
\hline
6 & 2 & 6 & 2 & 30 & 720 & 2 & 1440 & 15  \\
\hline
\end{tabular}
 
\vskip 10pt

\end{center}
 
The rank of the group is $6$ and so for Step 1, 
we have to consider $R=6\ccd 2$ (the case $R=1$ is obvious). 
To carry out Algorithm \ref{algo:step1-main}, 
we have to change the values of $N$ to $30$, 
$R:=6\ccd 2$. For $R=6$, we have to define $M:=2035800$ 
and Algorithm \ref{algo:step1-main} works
assuming that the action of $\weyl$ is recorded as 
$W_j$ ($j=1\ccd L$). The situation is similar for other values 
of $R$.

For the \pv{} (1), we use the following constants
to record the action of $\weyl$. 
Since all factors of $V$ are standard \rep s, 
we do not have to use the functions $\text{\bf CombN},\text{\bf NComb}$  
for Algorithms \ref{algo:programs-weyl-action54}, 
\ref{algo:programs-weyl-list}. 

\vskip 10pt

\begin{center}
 
\begin{tabular}{|c|c|c|c|c|c|c|c|c|}
\hline
$A$ & $B$ & $C$ & $L_1$ 
& $L_2$ & $L$ \\
\hline
3 & 2 & 18 & 6 & 2 & 72 \\
\hline
\end{tabular}
 
\vskip 10pt

\end{center}

\begin{algo}
\label{algo:programs-weyl-332}
\upshape
{Name:} ${\bf makeweyl332}(t_1,t_2,t_3,t_4)$ 

\vskip 5pt
{Require:} $t_1,t_2\in\gS_A,t_3\in\gS_B,t_4\in \gS_C$. 

{Description:} It makes $t_4$ the result of the action of 
$(t_1,t_2,t_3)\in\weyl$ on $\{\gam_1\ccd \gam_C\}$. 

{Local variables:}
$i,j,k\in \bbn$. 

\vskip 10pt
\noindent
1. For $i=1\ccd B$, $j=1\ccd A$, $k=1\ccd A$, 

\vskip 5pt
$t_4(9(i-1)+3(j-1)+k) \leftarrow 9(t_3(i)-1)+3(t_2(j)-1)+t_1(k)$.        

\vskip 5pt
This is the end of the function. 
\hfill $\diamond$
\end{algo}

It is easy to make an algorithm for the \pv{} (1) similar to 
Algorithm \ref{algo:programs-weyl-list} and so we do not provide the
details. 

The rank of the group is $5$ and so for Step 1, 
we have to consider $R=5\ccd 2$.  
To carry out Algorithm \ref{algo:step1-main}, 
we have to change the value of $N$ to $18$, 
$R:=5\ccd 2$. For $R=5$, we have to define $M:=8568$ 
and Algorithm \ref{algo:step1-main} works 
assuming that the action of $\weyl$ is recorded as 
$W_j$ ($j=1\ccd L$). 
The situation is similar for other values 
of $R$.

For the \pv{} (4), we us the following constants 
to documents elements of $\weyl$.

\vskip 10pt

\begin{center}
 
\begin{tabular}{|c|c|c|c|c|c|c|c|c|}
\hline
$N$ & $R$ & $A$ & $C$ & $L$ \\
\hline
8 & 3 & 8 & 56 & 40320 \\
\hline
\end{tabular}

\end{center}
 
\vskip 10pt

Note that we have to set $N:=8,R:=3$ to use 
the functions ${\bf CombN},{\bf NComb}$. 

\begin{algo}
\label{algo:programs-weyl-actiontri8}
\upshape
{Name:} ${\bf makeweyltri8}(t_1,t_2)$ 

\vskip 5pt
{Require:} $t_1\in\gS_A,t_2\in\gS_C$. 

{Description:} It makes $t_2$ the result of the action of 
$t_1\in\weyl$ on $\{\gam_1\ccd \gam_C\}$. 

{Local variables:}
(i) $i,j,k\in \bbn$, 

\quad
(ii) $v=(v_1,v_2,v_3),w=(w_1,w_2,w_3)$: 
arrays of elements of $\bbn$. 

\vskip 10pt

For $i=1\ccd C$, do the following. 

\vskip 5pt
\noindent
1. ${\bf NComb}(i,v)$.

\vskip 5pt
\noindent
2. $w_1\leftarrow t_1(v_1),\; w_2\leftarrow t_1(v_2),\; w_3\leftarrow t_1(v_3)$.

\vskip 5pt
\noindent
3. $j \leftarrow {\bf CombN}(w)$, $t_2(i)\leftarrow j$. 

\vskip 5pt
This is the end of the function. 
\hfill $\diamond$
\end{algo}

It is easy to make an algorithm for the \pv{} (4) similar to 
Algorithm \ref{algo:programs-weyl-list} and so we do not provide the
details. 

The rank of the group is $7$ and so for Step 1, 
we have to consider $R=7\ccd 2$.  
To carry out Algorithm \ref{algo:step1-main}, 
we have to change the value of $N$ to $56$, 
$R:=7\ccd 2$. For $R=7$, we have to define $M:=231917400$  
and Algorithm \ref{algo:step1-main} works 
assuming that the action of $\weyl$ is recorded as 
$W_j$ ($j=1\ccd L$). 
The situation is similar for other values 
of $R$. 

\subsection{Step 2}

We now explain algorithms in Step 2 of Section \ref{sec:outline-program}.
We now have to be sensitive to 
the number of digits of integers. So one has to use 
a computer language which allows an arbitrary number of digits.

We assume that $\gam_1\ccd \gam_N$ (the weights of coordinates) 
are recorded in a file 
(see Section \ref{sec:outline-program} for the values of $\gam_1\ccd \gam_N$).  
Let $D$ be the number of entries of elements of $\gt^*$. 
Explicitly, $D=8,8,9,8$ 
for the \pv s (1)--(4) respectively.

We list some basic easy functions used in 
Steps 2,3. We do not provide the details of the algorithms.
As we stated in Step 1, we assume that 
the values of $N,R$ have to be set before algorithms are carried out.

\begin{algo}
\label{algo:programs-basic-easy} 
\upshape
{\bf (Basic easy functions)}

\vskip 5pt
\noindent
(i) Name: ${\bf sorder}(a)$

\hskip 4pt
{Require:} $a=(a_1\ccd a_D)$: an array of elements of $\Q$. 

\hskip 4pt
{Description:} It sorts $a$ to an element of $\gt^*_+$ 
and returns the result. For example,  
for the \pv s (1), for 
$a=(a_{11},a_{12},a_{13},a_{21},a_{22},a_{23},a_{31},a_{32})$, 
$b=(b_{11},b_{12},b_{13},b_{21},b_{22},b_{23},b_{31},b_{32})$
is obtained  by sorting $(a_{11},a_{12},a_{13})$, 
$(a_{21},a_{22},a_{23})$ and $(a_{31},a_{32})$ 
so that $b_{11}\leq b_{12}\leq b_{13}$, 
$b_{21}\leq b_{22}\leq b_{23}$ and 
$b_{31}\leq b_{32}$.
Other cases are similar. This is a variation of 
the standard sorting algorithm.

\vskip 5pt
\noindent
(ii) Name: ${\bf veq}(a,b)$

\hskip 7pt
{Require:} $a=(a_1\ccd a_D),b=(b_1\ccd b_D)$:  
arrays of elements of $\Q$. 

\hskip 7pt
{Description:} It returns the value $0$ if $a=b$ and $1$ otherwise.

\vskip 5pt
\noindent
(iii) {Name:} ${\bf allpositive}(x)$ 

\hskip 9pt
{Require:} $x=(x_1\ccd x_R)$: an array of elements of $\Q$. 

\hskip 9pt
{Description:} It returns the value $1$ if $x_i>0$ for all $i$ 
and $0$ otherwise. 
\hfill $\diamond$
\end{algo}

We denote the zero  vector in $\gt^*$ by $Z$ as follows:
\begin{equation}
\label{eq:zero-vector}
Z=(\overbrace{0\ccd 0}^D) \in \gt^*.
\end{equation}
We shall use this vector in Step 3.

We now explain non-trivial functions 
needed in the main algorithm of Step 2. 
Let $r$ be the rank of the group $G$ 
and we fix $R=2\ccd r$. Let $A_R$ be the set 
defined in Section \ref{sec:outline-program}. 
We assume that we found a set of representatives $C'_R$ 
of $\weyl \backslash A_R$ in Step 1. 
Let $Q=\# C'_R$.  For example, it turns out 
that $Q=7891$ if $R=7$ for the \pv{} (3). 
We assume that $C'_R$ is documented in a file.
For $i=1\ccd Q$, let $v_i\in\gt^*$ be the 
$i$-th element of $C'_R$. Note that $v_i=(v_{i,j})$ 
is an array of $R$ distinct integers from $\{1\ccd N\}$ 
($N=\dim V$).   

For each $i$, we find a point of of form 
$\be'=c_1 \gam_{v_{i,1}}+\cdots + c_R \gam_{v_{i,R}}$
where $c_1+\cdots+c_R=1$, $c_1\ccd c_R>0$ and 
$\be'$ is orthogonal to vectors 
$\gam_{v_{i,j}}-\gam_{v_{i,1}}$ ($j=2\ccd R$). 
So if we put $m_{1,k}=1$ for $k=1\ccd R$, 
$m_{j,k}=(\gam_{v_{i,j}}-\gam_{v_{i,1}},\gam_{v_{i,k}})_*$ 
for $j=2\ccd R,k=1\ccd R$ and $m=(m_{j,k})$, then $c_1\ccd c_R$ have to satisfy the 
condition:
\begin{equation}
\label{eq:matrix-m}
m \begin{pmatrix}
c_1 \\ \vdots \\ c_R
\end{pmatrix} 
= \begin{pmatrix}
1 \\ \vdots \\ 0
\end{pmatrix}. 
\end{equation}

We augment $m$ by the vector $[1,0\ccd 0]$. 
As we discussed in Section \ref{sec:outline-program}, 
we only have to consider $\be'$ such that 
the above matrix $(m_{j,k})$ is non-singular,  
i.e., the $(R,R)$-entry of the reduced row echelon form 
of this augmented matrix is $1$.  If so,  
the last column is $[c_1\ccd c_R]$ except that 
the positivity of $c_1\ccd c_R$ has to be checked later.

\begin{algo}
{\bf (Functions in the main algorithm of Step 2)}
\label{algo:programs-mforclosest}
\upshape

\vskip 5pt

In the following functions, $R,D$ have to be defined and 
$\gam_1\ccd \gam_N$, 
$v_i=(v_{i,j})$ ($i=1\ccd Q$) from Step 1 have to be read from a file. 

\vskip 5pt
\noindent
(i) {Name:} ${\bf mforclosest}(m,a)$ 

\hskip 7pt
{Require:} (i) $m$: an $R\times (R+1)$ matrix with entries in $\Q$.

\quad\quad
(ii) $a=(a_1\ccd a_R)$: an array of elements of $\bbn$.  

\hskip 7pt
{Description:} 
It makes $m$ the matrix in (\ref{eq:matrix-m}) augmented by the right-hand side if 

\quad\quad
$a$ is substituted by $v_i$.

\hskip 7pt
{Local variables:} $j,k,l\in\bbn$.

\vskip 10pt
\noindent
1. For $j=2\ccd R$ and $k=1\ccd R$, 
$m_{j,k}\leftarrow \sum_{l=1}^D (\gam_{a_j,l}-\gam_{a_1,l})\gam_{a_k,l}$. 

\vskip 5pt
\noindent
2. For $k=1\ccd R+1$, $m_{1,k}\leftarrow 1$. 

\vskip 5pt
\noindent
3. For $j=2\ccd R$, $m_{j,R+1}\leftarrow 0$.

\vskip 5pt
\noindent
(ii) {Name:} ${\bf betacoefficient}(x,a)$

\hskip 9pt
{Require:} (i) $x=(x_1\ccd x_R)$: an array of elements of $\Q$.

\quad\quad
(ii) $a=(a_1\ccd a_R)$: an array of elements of $\bbn$.   

\hskip 9pt
{Description:} If $a$ is substituted by $v_i$ then 
it makes $x$ the unique solution to (\ref{eq:matrix-m}) 

\quad\quad
if the matrix $(m_{j,k})_{1\leq j,k\leq R}$ is non-singular
and the zero vector otherwise. 

\hskip 9pt
{Local variables:}
(i) $i\in\bbn$.  

\quad\quad
(ii) $y=(y_{j,k}),z=(z_{j,k})$:  $R\times (R+1)$ matrices 
with entries in $\Q$.

\vskip 10pt
\noindent
1. ${\bf mforclosest}(y,a)$.

\vskip 5pt
\noindent
2. $z\leftarrow$ the reduced row echelon form of $y$. 

\vskip 5pt
\noindent
3. If $z_{R,R}=1$ then for $i=1\ccd R$, $x_i\leftarrow z_{i,R+1}$. 
Otherwise for $i=1\ccd R$, $x_i\leftarrow 0$.  

\vskip 5pt
This is the end of the functions.
\hfill $\diamond$ 
\end{algo}

With these preparations, we can now describe the
main algorithm of Step 2. 
We consider the \pv{} (3) and the case $R=7$. 
We use the following constants

\vskip 10pt

\begin{center}
 
\begin{tabular}{|c|c|c|c|c|c|c|c|c|}
\hline
$N$ & $R$ & $D$ & $Q$ \\
\hline
40 & 7 & 9 & 7891  \\
\hline
\end{tabular}

\end{center}
 
\vskip 10pt
For other cases, these constants have to be changed appropriately. 

It will be convenient to use a computer language which 
has a linear algebra package with capability of
computing the reduced row echelon from 
(rref) of a matrix. 

\begin{algo}
\label{algo:programs-main-step2}
\upshape
{\bf (The main algorithm of Step 2)}

{Description:} 
We assume that elements of $C'_R$ are read from a file as 
$v_i=(v_{i,j})$ ($i=1\ccd Q$).  
For all $v_i$ ($i=1\ccd Q$), this algorithm finds the closest point $\be'$ 
of the convex hull of $\{\gam_{v_{i,1}}\ccd \gam_{v_{i,R}}\}$ to the 
origin if the convex hull is an $(R-1)$-dimensional simplex
and assigns the zero vector otherwise. 
We assume that constants $N,R,D,Q$ are defined as above and 
the function {\bf sorder} in Algorithm  
\ref{algo:programs-basic-easy} and functions in  
Algorithm \ref{algo:programs-mforclosest} 
are defined before this algorithm. 

{Local  variables:} 
(i) $i,j,k,l,m\in\bbn$.

\quad
(ii) $a=(a_1\ccd a_R)$: an array of elements of $\bbn$.  

\quad 
(iii) $x=(x_1\ccd x_R)$: an array of elements of $\Q$. 

\quad
(iv) $y=(y_1\ccd y_D),z=(z_1\ccd z_D)$: arrays of elements of $\Q$.

\quad
(v) $B^{(1)}=(B^{(1)}_{i,j}),C=(C_{i,j})$: 
$Q\times (D+2R)$ matrices with entries in $\Q$.

\vskip 10pt
\noindent
1. Initialize $m\leftarrow 0$. 

\vskip 5pt
\noindent
2. For $i=1\ccd Q$, do the following. 

\vskip 5pt
2-a. $a\leftarrow v_i$.

\vskip 5pt
2-b. ${\bf betacoefficient}(x,a)$. 

\vskip 5pt
2-c. If ${\bf allpositive}(x)=1$ then do the following.

\vskip 5pt
\quad\quad
2-c-1. $m\leftarrow m+1$. 

\vskip 5pt
\quad\quad
2-c-2. For $j=1\ccd R$, $C_{m,D+j}\leftarrow x_j$, $C_{m,D+R+j}\leftarrow a_j$. 

\vskip 5pt
\quad\quad
2-c-3. For $j=1\ccd D$, 
$C_{m,j} \leftarrow \sum_{l=1}^R C_{m,D+l} \gam_{C_{m,D+R+l},j}$.

%\vskip 5pt
%2-f. $i\leftarrow i+1$ and go back to 2-a. 

\vskip 5pt
In this step, we find $\be'=\sum_{k=1}^R x_k\gam_{v_{i,k}}$ where 
$x_1+\cdots+x_R=1$ and $\be'$ is orthogonal to
$\gam_{v_{i,k}}-\gam_{v_{i,1}}$ for $k=2\ccd R$ if 
such $\be'$ is unique.  If moreover $x_1\ccd x_R>0$ then 
we record entries of $\be'$ as $C_{m,1}\ccd C_{m,D}$, 
$x_1\ccd x_R$ as $C_{m,D+1}\ccd C_{m,D+R}$ 
and  $v_{i,1}\ccd v_{i,R}$ as 
$C_{m,D+R+1}\ccd C_{m,D+2R}$. 
The variable $m$ counts the number of vectors $\be'$ 
which are recorded. We record the last $2R$-entries 
of $C_{m,*}$ so that we know from which 
coordinates $\be'$ is made.   

We now move $\be'$ to an element $\be$ of $\gt^*_+$ 
as follows. It turns out that $m=343$ at this point. 

\vskip 5pt
\noindent
3. For $i=1\ccd m$ do the following. 

\vskip 5pt
3-a. For $j=1\ccd D$, $y_j\leftarrow C_{i,j}$. 

\vskip 5pt
3-b. $z\leftarrow {\bf sorder}(y)$. 

\vskip 5pt
3-c. For $j=1\ccd D$, $B^{(1)}_{i,j}\leftarrow z_j$. 

\vskip 5pt
3-d. For $j=1\ccd 2R$, $B^{(1)}_{i,D+j}\leftarrow C_{i,D+j}$. 

\vskip 5pt
\noindent
4 Record the first $m$ rows of $B^{(1)}$ in a file. 

\vskip 5pt
This finishes the algorithm for given $R$. 
\hfill $\diamond$
\end{algo}

When we consider the \pv{} (3), we consider $R=7\ccd 2$.
If we combine informations for all $R$, it is convenient
to change the step 1 so that 
the data for $B^{(1)}$ are without interruption. 
For example, 
when we consider the case $R=6$, we change 
the step 1 to $m\leftarrow 343$ since the number of $\be$
for $R=7$ is $343$. We add the obvious choice
for the case $R=1$ in the end. 
Even though the sizes of the columns are different
for different $R$'s, we only use the first $D$-columns
from now on and so it will not cause any problem.

We remove duplication from $B^{(1)}$ in Step 3.

\subsection{Step 3}

We assume that vectors $\be\in\gt^*_+$ from Step 2 
are read from a file as $B^{(1)}=(B^{(1)}_{i,j})$. 
As we pointed out above, we only use the first $D$-columns. 
Note that $[B^{(1)}_{i,1}\ccd B^{(1)}_{i,D}]$ is the 
$i$-th $\be\in \gt^*_+$. 
We also assume that the weight vectors $\gam_1\ccd \gam_N$
and functions in Step 2 are read from a file. 
Note that $Z=[0\ccd 0]$ is the zero vector in $\gt^*$. 

Let $Q$ be the number of $\be$ (in other words $i$'s in $B^{(1)}$). 
In the actual program, we removed the duplication in Step 2 
and moved to Step 3 and so we do not know this value $Q$. 
However, including the process of removing duplication twice
probably makes the reader slightly confusing. 
So we explain the process of removing the duplication only in
Step 3. We use the constants $N=\dim V,D,Q$. 
For example, $N=40,D=9$ for the \pv{} (3).

\vskip 10pt

\begin{algo}
\label{algo:programs-main-step3}
\upshape
{\bf (The main algorithm of Step 3)}
{Description:} 
This algorithm removes duplication and the zero vector 
from the list of $\be$. Then it finds coordinate vectors 
contained in the subspaces $Z_{\be},W_{\be}$. 
Constants $N,D,Q$ have to be defined as above.  
The function {\bf veq} in Algorithm \ref{algo:programs-basic-easy}
and the zero vector $Z$ (see (\ref{eq:zero-vector}))
have to be defined.

{Local variables:}
(i) $i,j,k,l,m,p\in\bbn$. 

\quad
(ii) $y=(y_1\ccd y_D),z=(z_1\ccd z_D)$: arrays of elements of $\Q$.   

\quad
(iii) $B^{(2)}=(B^{(2)}_{i,j})$: a $Q\times D$ matrix with entries in $\Q$. 

\quad
(iv) $E=(E_{i,j})$: a $Q\times N$ matrix with entries in $\bbn$. 

\vskip 10pt
\noindent
1. Initialize $m\leftarrow 0$. 

\vskip 5pt
\noindent
2. For $i=1\ccd Q$, do the following. 

\vskip 5pt
2-a. $l\leftarrow 1$. 

\vskip 5pt
2-b. For $k=1\ccd D$, $y_k\leftarrow B^{(1)}_{i,k}$. 

\vskip 5pt
2-c. For $j=1\ccd i-1$, do the following. 

\vskip 5pt
\quad\quad
2-c-1. For $k=1\ccd D$, 
$z_k\leftarrow B^{(1)}_{j,k}$.

\vskip 5pt
\quad\quad
2-c-2. $l\leftarrow {\bf veq}(y,z) l$. 

%\vskip 5pt
%\quad\quad
%2-c-3. $j\leftarrow j+1$ and go back to 2-b-1. 

\vskip 5pt
2-d. $l\leftarrow {\bf veq}(y,Z) l$. 

\vskip 5pt
2-e. If $l=1$ then $m\leftarrow m+1$ and 
for $k=1\ccd D$, $B^{(2)}_{m,k}\leftarrow B^{(1)}_{i,k}$. 

%\vskip 5pt
%2-g. $i\leftarrow i+1$ and go back to 2-a. 

\vskip 5pt

The steps 2-a\ccd 2-e remove the duplication 
and the zero vector.  We now determine coordinate 
vectors contained in $Z_{\be},W_{\be}$. 

\vskip 5pt
\noindent
3. For $i=1\ccd m$, do the following.

\vskip 5pt
3-a. For $j=1\ccd N$, do the following. 

\vskip 5pt
\quad\quad
3-a-1. $k\leftarrow \sum_{p=1}^D B^{(2)}_{i,p} \gam_{j,p}$.

\vskip 5pt
\quad\quad
3-a-2. $l\leftarrow \sum_{p=1}^D (B^{(2)}_{i,p})^2$.

\vskip 5pt
\quad\quad
3-a-3. If $k>l$ then $E_{i,j}\leftarrow 2$.

\vskip 5pt
\quad\quad
3-a-4. If $k=l$ then $E_{i,j}\leftarrow 1$.

\vskip 5pt
\quad\quad
3-a-5. If $k<l$ then $E_{i,j}\leftarrow 0$.

\vskip 5pt
\noindent
4. Record the first $m$ rows of $B^{(2)}$ and $E$. 

\vskip 5pt
This finishes the algorithm. 
\hfill $\diamond$
\end{algo}

It turns out that after the step 2 above,  
$m=49,81,292,183$ for the \pv s (1)--(4) 
respectively. $E$ is a matrix which determines the subspaces 
$Z_{\be},W_{\be}$. For $\beta\in\gB$, the $j$-th coordinate 
vector $\mathbbm a_j$ belongs to $Z_{\be}$ (resp. $W_{\be}$)
if and only if 
\begin{equation*}
(\be,\gam_j)_*=(\be,\be)_* \quad   
(\text{resp. } (\be,\gam_j)_* > (\be,\be)_*). 
\end{equation*}
Note that $Z_{\be},W_{\be}$ are spanned by coordinate vectors 
contained in them. We substituted $2,1$ to $E_{i,j}$ 
according as $\mathbbm a_j\in W_{\be},Z_{\be}$ 
and $0$ otherwise in the step 3.

Algorithms \ref{algo:programs-main-step2}, \ref{algo:programs-main-step3}
are the same for the \pv s (1)--(4). One just has to 
include correct informations such as $\gam_1\ccd \gam_N$, $N,R,D$, etc.

\section{Output for the case (1)}
\label{sec:output1}

In this section we list the output of our programming
for the case (1).  We made the program so that the output 
will be a tex file. 
Let $\gt^*_+$ be the Weyl chamber and 
$\mathbbm a_1\ccd \mathbbm a_{18}$ the coordinate 
vectors both defined in Section \ref{sec:outline-program}. 
The set $\gB$ consists of $\be_i$'s in the following table.  
Note that $\be_i\in\gt^*_+$ for all $i$. 

\vskip10pt

\begin{center}

\begin{tabular}{|l|l|l|}
\hline
$\; $\hskip 0.9in \tiny $\beta$ \hskip 0.9in \normalsize
& $\;$ \hskip 0.1in \tiny $i$ such that $\mathbbm a_i\in Z_{\be}$ 
\normalsize \hskip 0.1in $\;$
& $\;$ \hskip 0.1in \tiny $i$ such that $\mathbbm a_i\in W_{\be}$ 
\normalsize \hskip 0.1in $\;$ \\
\hline
\tiny $\be_1 = \tfrac {1} {42} (-2,-2,4,0,0,0,-3,3)$ \normalsize 
& \tiny $7,8,9,10,11,12,13,14,15 $ \normalsize 
& \tiny $16,17,18 $ \normalsize \\ 
\hline 
\tiny $\be_2 = \tfrac {1} {22} (-4,2,2,-2,-2,4,-3,3)$ \normalsize 
& \tiny $6,9,12,13,14,16,17 $ \normalsize 
& \tiny $15,18 $ \normalsize \\ 
\hline 
\tiny $\be_3 = \tfrac {1} {22} (-2,-2,4,-4,2,2,-3,3)$ \normalsize 
& \tiny $8,9,11,12,14,15,16 $ \normalsize  
& \tiny $17,18 $ \normalsize \\ 
\hline 
\tiny $\be_4 = \tfrac {1} {6} (-2,1,1,0,0,0,0,0)$ \normalsize 
& \tiny $4,5,6,7,8,9,13,14,15,16,17,18 $ \normalsize  
& \tiny $\;$ \hskip 0.3in - \normalsize \\ 
\hline 
\tiny $\be_5 = \tfrac {1} {6} (0,0,0,-2,1,1,0,0)$ \normalsize 
& \tiny $2,3,5,6,8,9,11,12,14,15,17,18 $ \normalsize  
& \tiny $\;$ \hskip 0.3in - \normalsize \\ 
\hline 
\tiny $\be_6 = \tfrac {1} {66} (-4,2,2,-4,2,2,-3,3)$ \normalsize 
& \tiny $5,6,8,9,11,12,13,16 $ \normalsize 
& \tiny $14,15,17,18 $ \normalsize \\ 
\hline 
\tiny $\be_7 = \tfrac {1} {42} (0,0,0,-2,-2,4,-3,3)$ \normalsize 
& \tiny $3,6,9,10,11,13,14,16,17 $ \normalsize  
& \tiny $12,15,18 $ \normalsize \\
\hline 
\end{tabular}

\end{center}

\begin{center}

\begin{tabular}{|l|l|l|}
\hline
$\; $\hskip 0.9in \tiny $\beta$ \hskip 0.9in \normalsize
& $\;$ \hskip 0.1in \tiny $i$ such that $\coorde_i\in Z_{\be}$ 
\normalsize \hskip 0.1in $\;$ 
& $\;$ \hskip 0.1in \tiny $i$ such that $\coorde_i\in W_{\be}$ 
\normalsize \hskip 0.1in $\;$ \\
\hline
\tiny $\be_{33} = \tfrac {1} {66} (-22,8,14,-4,-4,8,-3,3)$ \normalsize 
& \tiny $6,16,17 $ \normalsize  
& \tiny $9,15,18 $ \normalsize \\ 
\hline 
\tiny $\be_{34} = \tfrac {1} {6} (0,0,0,-2,-2,4,-3,3)$ \normalsize 
& \tiny $12,15,18 $ \normalsize  
& \tiny $\;$ \hskip 0.3in - \normalsize \\ 
\hline
\tiny $\be_{35} = \tfrac {1} {42} (-2,-2,4,-14,4,10,-21,21)$ \normalsize 
& \tiny $12,15,17 $ \normalsize  
& \tiny $18 $ \normalsize \\ 
\hline 
\tiny $\be_{36} = \tfrac {1} {42} (-2,-2,4,-14,-14,28,-3,3)$ \normalsize 
& \tiny $9,12,15 $ \normalsize 
& \tiny $18 $ \normalsize \\ 
\hline 
\tiny $\be_{37} = \tfrac {1} {66} (-4,-4,8,-22,8,14,-3,3)$ \normalsize 
& \tiny $8,12,15 $ \normalsize 
& \tiny $9,17,18 $ \normalsize \\ 
\hline 
\tiny $\be_{38} = \tfrac {1} {2} (0,0,0,0,0,0,-1,1)$ \normalsize 
& \tiny $10,11,12,13,14,15,16,17,18 $ \normalsize 
& \tiny $\;$ \hskip 0.3in - \normalsize \\ 
\hline 
\tiny $\be_{39} = \tfrac {1} {6} (-2,0,2,-2,0,2,-1,1)$ \normalsize 
& \tiny $9,15,17 $ \normalsize 
& \tiny $18 $ \normalsize \\
\hline 
\end{tabular}

\end{center}

\section{Output for the case (2)}
\label{sec:output2}

In this section we list the output of our programming
for the case (2).  
Let $\gt^*_+$ be the Weyl chamber and 
$\mathbbm a_1\ccd \mathbbm a_{30}$ the coordinate 
vectors both defined in Section \ref{sec:outline-program}. 
The set $\gB$ consists of $\be_i$'s in the following table.

\vskip 10pt

\begin{center}
 
\begin{tabular}{|l|l|l|}
\hline
$\; $\hskip 1in \tiny $\beta$ \hskip 1in \normalsize
& $\;$ \hskip 0.1in \tiny $i$ such that 
$\coorde_i\in Z_{\be}$ \hskip 0.1in $\;$ \normalsize
& $\;$ \hskip 0.1in \tiny $i$ such that 
$\coorde_i\in W_{\be}$ \hskip 0.1in $\;$ \normalsize \\
\hline
\tiny $\be_1=\tfrac {1} {48} (-1,-1,-1,-1,-1,5,-3,3)$ \normalsize 
& \tiny 
\begin{math}
\begin{matrix}
5,9,12,14,15,16,17,18,19,21, \\
22,23,25,26,28 \hfill 
\end{matrix}
\end{math}
\normalsize 
& \tiny $20,24,27,29,30 $ \normalsize \\ 
\hline 
\tiny $\be_2=\tfrac {1} {8} (-1,-1,-1,1,1,1,-1,1)$ \normalsize 
& \tiny 
\begin{math}
\begin{matrix}
13,14,15,18,19,20,22,23,24, \\
25,26,27 \hfill 
\end{matrix}
\end{math}
\normalsize 
& \tiny $28,29,30 $ \normalsize \\ 
\hline 
\tiny $\be_3=\tfrac {1} {15} (-5,1,1,1,1,1,0,0)$ \normalsize 
& \tiny 
\begin{math}
\begin{matrix}
6,7,8,9,10,11,12,13,14,15,21, \\
22,23,24,25,26,27,28,29,30 \hfill 
\end{matrix}
\end{math}
\normalsize 
& \tiny $\;$ \hskip 0.3in - \normalsize \\ 
\hline 
\tiny $\be_4=\tfrac {1} {120} (-7,-1,-1,-1,5,5,-3,3)$ \normalsize 
& \tiny 
\begin{math}
\begin{matrix}
8,9,11,12,13,14,19,20,21, \\
22,25 \hfill 
\end{matrix}
\end{math}
\normalsize
& \tiny $15,23,24,26,27,28,29,30 $ \normalsize \\ 
\hline 
\tiny $\be_5=\tfrac {7} {138} (-4,-4,-4,2,2,8,-3,3)$ \normalsize 
& \tiny $14,15,20,24,27,28 $ \normalsize 
& \tiny $29,30 $ \normalsize \\ 
\hline 
\tiny $\be_6=\tfrac {1} {138} (-4,-4,-4,2,2,8,-3,3)$ \normalsize 
& \tiny $5,9,12,13,18,19,22,23,25,26 $ \normalsize 
& \tiny $14,15,20,24,27,28,29,30 $ \normalsize \\ 
\hline 
\tiny $\be_7=\tfrac {2} {15} (-1,-1,-1,-1,-1,5,0,0)$ \normalsize 
& \tiny $5,9,12,14,15,20,24,27,29,30 $ \normalsize 
& \tiny $\;$ \hskip 0.3in - \normalsize \\
\hline

\end{tabular}

\end{center}

\newpage

\begin{center}
 
\begin{tabular}{|l|l|l|}
\hline
$\; $\hskip 1in \tiny $\beta$ \hskip 1in \normalsize
& $\;$ \hskip 0.1in \tiny $i$ such that 
$\coorde_i\in Z_{\be}$ \hskip 0.1in $\;$ \normalsize
& $\;$ \hskip 0.1in \tiny $i$ such that 
$\coorde_i\in W_{\be}$ \hskip 0.1in $\;$ \normalsize \\
\hline 
\tiny $\be_{17}=\tfrac {1} {42} (-4,-2,0,0,2,4,-1,1)$ \normalsize 
& \tiny $9,11,13,20,23,25 $ \normalsize  
& \tiny 
\begin{math}
\begin{matrix}
12,14,15,24,26,27,28, \\
29,30 \hfill 
\end{matrix}
\end{math}
\normalsize \\ 
\hline
\tiny $\be_{18}=\tfrac {1} {12} (-1,-1,-1,-1,2,2,0,0)$ \normalsize 
& \tiny 
\begin{math}
\begin{matrix}
4,5,8,9,11,12,13,14,19,20,\\
23,24,26,27,28,29 \hfill 
\end{matrix}
\end{math}
\normalsize 
& \tiny $15,30 $ \normalsize \\ 
\hline 
\tiny $\be_{19}=\tfrac {1} {30} (-4,-4,-4,-4,-4,20,-15,15)$ \normalsize 
& \tiny $20,24,27,29,30 $ \normalsize 
& \tiny $\;$ \hskip 0.3in - \normalsize \\ 
\hline 
\tiny $\be_{20}=\tfrac {1} {78} (-14,-14,-14,-14,4,52,-9,9)$ \normalsize 
& \tiny $15,20,24,27,29 $ \normalsize  
& \tiny $30 $ \normalsize \\ 
\hline 
\tiny $\be_{21}=\tfrac {1} {102} (-34,-4,-4,14,14,14,-9,9)$ \normalsize 
& \tiny $13,14,15,22,23,24,25,26,27 $ \normalsize 
& \tiny $28,29,30 $ \normalsize \\ 
\hline 
\tiny $\be_{22}=\tfrac {1} {30} (-4,-4,-4,2,5,5,-3,3)$ \normalsize 
& \tiny $13,14,19,20,23,24,26,27 $ \normalsize 
& \tiny $15,28,29,30 $ \normalsize \\ 
\hline
\tiny $\be_{23}=\tfrac {1} {42} (-8,-8,-8,-2,10,16,-9,9)$ \normalsize 
& \tiny $15,20,24,27,28 $ \normalsize 
& \tiny $29,30 $ \normalsize \\
\hline

\end{tabular}

\end{center}

\begin{center}

\begin{tabular}{|l|l|l|}
\hline
$\; $\hskip 1in \tiny $\beta$ \hskip 1in \normalsize
& $\;$ \hskip 0.1in \tiny $i$ such that 
$\coorde_i\in Z_{\be}$ \hskip 0.1in $\;$ \normalsize
& $\;$ \hskip 0.2in \tiny $i$ such that 
$\coorde_i\in W_{\be}$ \hskip 0.1in $\;$ \normalsize \\
\hline 
\tiny $\be_{57} = \tfrac {1} {24} (-8,-2,1,1,4,4,-3,3)$ \normalsize 
& \tiny $15,23,24,25 $ \normalsize 
& \tiny $26,27,28,29,30 $ \normalsize \\ 
\hline 
\tiny $\be_{58} = \tfrac {1} {78} (-14,-8,-8,4,10,16,-9,9)$ \normalsize 
& \tiny $14,20,23,26 $ \normalsize  
& \tiny $15,24,27,28,29,30 $ \normalsize \\ 
\hline 
\tiny $\be_{59} = \tfrac {1} {3} (-1,-1,0,0,0,2,0,0)$ \normalsize 
& \tiny $12,14,15,27,29,30 $ \normalsize  
& \tiny $\;$ \hskip 0.3in - \normalsize \\ 
\hline 
\tiny $\be_{60} = \tfrac {1} {3} (-1,-1,-1,1,1,1,0,0)$ \normalsize 
& \tiny $13,14,15,28,29,30 $ \normalsize  
& \tiny $\;$ \hskip 0.3in - \normalsize \\ 
\hline 
\tiny $\be_{61} = \tfrac {1} {21} (-7,-1,-1,2,2,5,0,0)$ \normalsize 
& \tiny $9,12,13,24,27,28 $ \normalsize 
& \tiny $14,15,29,30 $ \normalsize \\ 
\hline 
\tiny $\be_{62} = \tfrac {1} {12} (-4,-1,-1,-1,-1,8,0,0)$ \normalsize 
& \tiny $9,12,14,15,24,27,29,30 $ \normalsize 
& \tiny $\;$ \hskip 0.3in - \normalsize \\ 
\hline 
\tiny $\be_{63} = \tfrac {1} {30} (-10,-10,-1,5,5,11,-3,3)$ \normalsize 
& \tiny $14,15,27,28 $ \normalsize 
& \tiny $29,30 $ \normalsize \\
\hline
\end{tabular}

\end{center}

\section{Output for the case (3)}
\label{sec:output3}

In this section we list the output of our programming
for the case (3).  
Let $\gt^*_+$ be the Weyl chamber and 
$\mathbbm a_1\ccd \mathbbm a_{40}$ the coordinate 
vectors both defined in Section \ref{sec:outline-program}. 
The set $\gB$ consists of $\be_i$'s in the following table.  

\vskip 10pt

\begin{center}

\begin{tabular}{|l|l|}
\hline
$\; $\hskip 1in \tiny $\beta$ \hskip 1in $\;$ \normalsize
& $\;$ \hskip 1in \tiny $i$ such that 
$\coorde_i\in Z_{\be}$ \hskip 1in $\;$ \normalsize \\
\cline{2-2}
& $\;$ \hskip 1in \tiny $i$ such that 
$\coorde_i\in W_{\be}$ \hskip 1in $\;$ \normalsize \\
\hline
\hline 
\tiny $\be_1=$ \normalsize
& \tiny 
\begin{math}
\begin{matrix}
11,12,13,14,15,16,17,18,19,20,21,22,23,24,25, \\
26,27,28,29,30,31,32,33,34,35,36,37,38,39,40 
\end{matrix}
\end{math} \normalsize \\
\cline{2-2}
\tiny $\tfrac {1} {12} (0,0,0,0,0,-3,1,1,1)$ \normalsize 
& \tiny $\;$ \hskip 0.3in - \normalsize \\ 
\hline 
\hline
\tiny $\be_2=$ \normalsize & 
\tiny $4,7,9,10,14,17,19,20,24,27,29,30,31,32,33,35,36,38 $ \normalsize \\ 
\cline{2-2}
\tiny $\tfrac {7} {620} (-4,-4,-4,-4,16,-5,-5,-5,15)$ \normalsize 
& \tiny $34,37,39,40$ \normalsize \\ 
\hline
\hline
\tiny $\be_3=$ \normalsize 
& \tiny $8,9,10,12,13,14,15,16,17,22,23,24,25,26,27,32,33,34,35,36,37 $ \normalsize \\ 
\cline{2-2}
\tiny $\tfrac {1} {780} (-12,-12,8,8,8,-15,5,5,5)$ \normalsize 
& \tiny $18,19,20,28,29,30,38,39,40 $ \normalsize \\ 
\hline
\hline 
\tiny $\be_4=$ \normalsize 
& \tiny $7,9,10,17,19,20,27,29,30,34,35,36,38 $ \normalsize \\ 
\cline{2-2}
\tiny  $\tfrac {3} {44} (-4,0,0,0,4,-1,-1,-1,3)$ \normalsize 
& \tiny $37,39,40 $ \normalsize \\ 
\hline

\end{tabular}

\end{center}

\newpage

\begin{center}

\begin{tabular}{|l|l|}
\hline
$\; $\hskip 1in \tiny $\beta$ \hskip 1in $\;$ \normalsize
& $\;$ \hskip 1in \tiny $i$ such that 
$\coorde_i\in Z_{\be}$ \hskip 1in $\;$ \normalsize \\
\cline{2-2}
& $\;$ \hskip 1in \tiny $i$ such that 
$\coorde_i\in W_{\be}$ \hskip 1in $\;$ \normalsize \\
\hline
\hline 
\tiny $\be_5=$ \normalsize 
& \tiny $4,7,9,10,14,17,19,20,21,22,23,25,26,28,31,32,33,35,36,38 $ \normalsize \\ 
\cline{2-2}
\tiny $\tfrac {1} {180} (-2,-2,-2,-2,8,-5,-5,5,5)$ \normalsize 
& \tiny $24,27,29,30,34,37,39,40 $ \normalsize \\ 
\hline 
\hline
\tiny $\be_6=$ \normalsize 
& \tiny $7,9,10,17,19,20,24,25,26,28,34,35,36,38 $ \normalsize \\ 
\cline{2-2}
\tiny $\tfrac {1} {12} (-2,0,0,0,2,-1,-1,1,1)$ \normalsize 
& \tiny $27,29,30,37,39,40 $ \normalsize \\ 
\hline
\hline 
\tiny $\be_7=$ \normalsize 
& \tiny $7,9,10,14,15,16,18,24,25,26,28,34,35,36,38 $ \normalsize \\ 
\cline{2-2}
\tiny $\tfrac {1} {44} (-4,0,0,0,4,-3,1,1,1)$ \normalsize 
& \tiny $17,19,20,27,29,30,37,39,40 $ \normalsize \\
\hline
\end{tabular}

\end{center}

\begin{center}

\begin{tabular}{|l|l|}
\hline
$\; $\hskip 1in \tiny $\beta$ \hskip 1in $\;$ \normalsize
& $\;$ \hskip 1in \tiny $i$ such that 
$\coorde_i\in Z_{\be}$ \hskip 1in $\;$ \normalsize \\
\cline{2-2}
& $\;$ \hskip 1in \tiny $i$ such that 
$\coorde_i\in W_{\be}$ \hskip 1in $\;$ \normalsize \\
\hline 
\hline
\tiny $\be_{25}=$ \normalsize 
& \tiny $10,20,23,24,25,31,32 $ \normalsize \\ 
\cline{2-2}
\tiny $\tfrac {7} {3820} (-24,-4,-4,16,16,-25,-25,15,35)$ \normalsize 
& \tiny $26,27,28,29,30,33,34,35,36,37,38,39,40 $ \normalsize \\ 
\hline 
\hline
\tiny $\be_{26}=$ \normalsize 
& \tiny $9,10,14,17,18,24,27,28,31 $ \normalsize \\ 
\cline{2-2}
\tiny $\tfrac {3} {3020} (-16,-16,4,4,24,-25,-5,-5,35)$ \normalsize 
& \tiny $19,20,29,30,32,33,34,35,36,37,38,39,40 $ \normalsize \\ 
\hline 
\hline
\tiny $\be_{27}=$ \normalsize 
& \tiny $7,8,17,18,27,28,34,35,36 $ \normalsize \\ 
\cline{2-2}
\tiny $\tfrac {11} {2380} (-32,-12,8,8,28,-5,-5,-5,15)$ \normalsize 
& \tiny $9,10,19,20,29,30,37,38,39,40 $ \normalsize \\
\hline 
\end{tabular}

\end{center}

\begin{center}

\begin{tabular}{|l|l|}
\hline
$\; $\hskip 1in \tiny $\beta$ \hskip 1in $\;$ \normalsize
& $\;$ \hskip 1in \tiny $i$ such that 
$\coorde_i\in Z_{\be}$ \hskip 1in $\;$ \normalsize \\
\cline{2-2}
& $\;$ \hskip 1in \tiny $i$ such that 
$\coorde_i\in W_{\be}$ \hskip 1in $\;$ \normalsize \\
\hline 
\hline
\tiny $\be_{45}=$ \normalsize 
& \tiny $10,16,17,18,19,23,24,25,31,32 $ \normalsize \\ 
\cline{2-2}
\tiny $\tfrac {1} {780} (-12,-2,-2,8,8,-15,-5,5,15)$ \normalsize 
& \tiny $20,26,27,28,29,30,33,34,35,36,37,38,39,40 $ \normalsize \\ 
\hline 
\hline
\tiny $\be_{46}=$ \normalsize 
& \tiny $10,17,19,24,26,28,33,35 $ \normalsize \\ 
\cline{2-2}
\tiny $\tfrac {7} {1020} (-14,-4,-4,6,16,-15,-5,5,15)$ \normalsize 
& \tiny $20,27,29,30,34,36,37,38,39,40 $ \normalsize \\ 
\hline 
\hline
\tiny $\be_{47}=$ \normalsize 
& \tiny $7,8,17,18,24,25,26,32,33 $ \normalsize \\ 
\cline{2-2}
\tiny $\tfrac {1} {3180} (-32,-12,8,8,28,-15,-15,5,25)$ \normalsize 
& \tiny $9,10,19,20,27,28,29,30,34,35,36,37,38,39,40 $ \normalsize \\
\hline 
\end{tabular}

\end{center}

\begin{center}

\begin{tabular}{|l|l|}
\hline
$\; $\hskip 1in \tiny $\beta$ \hskip 1in $\;$ \normalsize
& $\;$ \hskip 1in \tiny $i$ such that 
$\coorde_i\in Z_{\be}$ \hskip 1in $\;$ \normalsize \\
\cline{2-2}
& $\;$ \hskip 1in \tiny $i$ such that 
$\coorde_i\in W_{\be}$ \hskip 1in $\;$ \normalsize \\
\hline 
\hline
\tiny $\be_{65}=$ \normalsize 
& \tiny $7,9,10,17,19,20,24,35,36,38 $ \normalsize \\ 
\cline{2-2}
\tiny $\tfrac {1} {240} (-31,-6,-6,-6,49,-20,-20,5,35)$ \normalsize 
& \tiny $27,29,30,34,37,39,40 $ \normalsize \\ 
\hline 
\hline
\tiny $\be_{66}=$ \normalsize 
& \tiny $7,9,10,17,19,20,25,26,28,31,32,33 $ \normalsize \\ 
\cline{2-2}
\tiny $\tfrac {1} {460} (-34,6,6,6,16,-15,-15,-5,35)$ \normalsize 
& \tiny $27,29,30,34,35,36,37,38,39,40 $ \normalsize \\ 
\hline 
\hline
\tiny $\be_{67}=$ \normalsize 
& \tiny $10,20,24,27,29,31,32,35 $ \normalsize \\ 
\cline{2-2}
\tiny $\tfrac {1} {540} (-26,-26,-26,-16,94,-35,-35,-25,95)$ \normalsize 
& \tiny $30,33,34,36,37,38,39,40 $ \normalsize \\
\hline  
\end{tabular}

\end{center}

\begin{center}

\begin{tabular}{|l|l|}
\hline
$\; $\hskip 1in \tiny $\beta$ \hskip 1in $\;$ \normalsize
& $\;$ \hskip 0.9in \tiny $i$ such that 
$\coorde_i\in Z_{\be}$ \hskip 0.9in $\;$ \normalsize \\
\cline{2-2}
& $\;$ \hskip 0.9in \tiny $i$ such that 
$\coorde_i\in W_{\be}$ \hskip 0.9in $\;$ \normalsize \\
\hline 
\hline
\tiny $\be_{85}=$ \normalsize 
& \tiny $20,24,26,28,33,35 $ \normalsize \\ 
\cline{2-2}
\tiny $\tfrac {1} {740} (-56,-16,-16,24,64,-185,-5,75,115)$ \normalsize 
& \tiny $27,29,30,34,36,37,38,39,40 $ \normalsize \\ 
\hline 
\hline
\tiny $\be_{86}=$ \normalsize 
& \tiny $17,18,27,28,34,35,36 $ \normalsize \\ 
\cline{2-2}
\tiny $\tfrac {1} {440} (-56,-21,14,14,49,-110,25,25,60)$ \normalsize 
& \tiny $19,20,29,30,37,38,39,40 $ \normalsize \\ 
\hline 
\hline
\tiny $\be_{87}=$ \normalsize 
& \tiny $17,18,24,25,26,34,35,36 $ \normalsize \\ 
\cline{2-2}
\tiny $\tfrac {1} {220} (-8,-3,2,2,7,-55,15,20,20)$ \normalsize 
& \tiny $19,20,27,28,29,30,37,38,39,40 $ \normalsize \\
\hline
\end{tabular}

\end{center}

\begin{center}

\begin{tabular}{|l|l|}
\hline
$\; $\hskip 1in \tiny $\beta$ \hskip 1in $\;$ \normalsize
& $\;$ \hskip 1in \tiny $i$ such that 
$\coorde_i\in Z_{\be}$ \hskip 1in $\;$ \normalsize \\
\cline{2-2}
& $\;$ \hskip 1in \tiny $i$ such that 
$\coorde_i\in W_{\be}$ \hskip 1in $\;$ \normalsize \\
\hline
\hline
\tiny $\be_{105}=$ \normalsize 
& \tiny $10,16,17,18,19,26,27,28,29,36,37,38,39 $ \normalsize \\ 
\cline{2-2}
\tiny $\tfrac {1} {140} (-56,4,4,24,24,-15,5,5,5)$ \normalsize 
& \tiny $20,30,40 $ \normalsize \\ 
\hline 
\hline
\tiny $\be_{106}=$ \normalsize 
& \tiny $16,17,18,19,26,27,28,29,33,34,35 $ \normalsize \\ 
\cline{2-2}
\tiny $\tfrac {1} {260} (-24,-4,-4,16,16,-65,15,15,35)$ \normalsize 
& \tiny $20,30,36,37,38,39,40 $ \normalsize \\ 
\hline 
\hline
\tiny $\be_{107}=$ \normalsize 
& \tiny $10,16,17,18,19,26,27,28,29,31,32 $ \normalsize \\ 
\cline{2-2}
\tiny $\tfrac {1} {60} (-4,0,0,2,2,-3,-1,-1,5)$ \normalsize 
& \tiny $20,30,33,34,35,36,37,38,39,40 $ \normalsize \\
\hline 
\end{tabular}

\end{center}

\begin{center}

\begin{tabular}{|l|l|}
\hline
$\; $\hskip 1in \tiny $\beta$ \hskip 1in $\;$ \normalsize
& $\;$ \hskip 1in \tiny $i$ such that 
$\coorde_i\in Z_{\be}$ \hskip 1in $\;$ \normalsize \\
\cline{2-2}
& $\;$ \hskip 1in \tiny $i$ such that 
$\coorde_i\in W_{\be}$ \hskip 1in $\;$ \normalsize \\
\hline 
\hline
\tiny $\be_{125}=$ \normalsize 
& \tiny $9,10,17,24,28,34,38 $ \normalsize \\ 
\cline{2-2}
\tiny $\tfrac {1} {440} (-71,-56,14,14,99,-60,10,25,25)$ \normalsize 
& \tiny $19,20,27,29,30,37,39,40 $ \normalsize \\ 
\hline 
\hline
\tiny $\be_{126}=$ \normalsize 
& \tiny $10,16,18,24,33,35 $ \normalsize \\ 
\cline{2-2}
\tiny $\tfrac {1} {1660} (-124,-24,-24,76,96,-135,-15,65,85)$ \normalsize 
& \tiny $17,19,20,26,27,28,29,30,34,36,37,38,39,40 $ \normalsize \\ 
\hline 
\hline
\tiny $\be_{127}=$ \normalsize 
& \tiny $7,8,17,18,25,26,34 $ \normalsize \\ 
\cline{2-2}
\tiny $\tfrac {3} {520} (-16,-1,4,4,9,-5,-5,0,10)$ \normalsize 
& \tiny $9,10,19,20,27,28,29,30,35,36,37,38,39,40 $ \normalsize \\
\hline 
\end{tabular}

\end{center}

\begin{center}

\begin{tabular}{|l|l|}
\hline
$\; $\hskip 1in \tiny $\beta$ \hskip 1in $\;$ \normalsize
& $\;$ \hskip 1in \tiny $i$ such that 
$\coorde_i\in Z_{\be}$ \hskip 1in $\;$ \normalsize \\
\cline{2-2}
& $\;$ \hskip 1in \tiny $i$ such that 
$\coorde_i\in W_{\be}$ \hskip 1in $\;$ \normalsize \\
\hline 
\hline
\tiny $\be_{145}=$ \normalsize 
& \tiny $9,17,18,27,28,33,35 $ \normalsize \\ 
\cline{2-2}
\tiny $\tfrac {1} {340} (-36,-16,-1,19,34,-20,-5,-5,30)$ \normalsize 
& \tiny $10,19,20,29,30,34,36,37,38,39,40 $ \normalsize \\ 
\hline 
\hline
\tiny $\be_{146}=$ \normalsize 
& \tiny $9,10,17,18,27,28,34,35,36 $ \normalsize \\ 
\cline{2-2}
\tiny $\tfrac {1} {45} (-8,-3,2,2,7,-5,0,0,5)$ \normalsize 
& \tiny $19,20,29,30,37,38,39,40 $ \normalsize \\ 
\hline 
\hline
\tiny $\be_{147}=$ \normalsize 
& \tiny $10,20,26,27,28,29,36,37,38,39 $ \normalsize \\ 
\cline{2-2}
\tiny $\tfrac {1} {40} (-16,-1,-1,9,9,-5,-5,5,5)$ \normalsize 
& \tiny $30,40 $ \normalsize \\
\hline 
\end{tabular}

\end{center}

\begin{center}

\begin{tabular}{|l|l|}
\hline
$\; $\hskip 1in \tiny $\beta$ \hskip 1in $\;$ \normalsize
& $\;$ \hskip 0.9in \tiny $i$ such that 
$\coorde_i\in Z_{\be}$ \hskip 0.9in $\;$ \normalsize \\
\cline{2-2}
& $\;$ \hskip 0.9in \tiny $i$ such that 
$\coorde_i\in W_{\be}$ \hskip 0.9in $\;$ \normalsize \\
\hline 
\hline
\tiny $\be_{165}=$ \normalsize 
& \tiny $26,27,28,29,33,34,35 $ \normalsize \\ 
\cline{2-2}
\tiny $\tfrac {1} {660} (-24,-4,-4,16,16,-165,-165,155,175)$ \normalsize 
& \tiny $30,36,37,38,39,40 $ \normalsize \\ 
\hline 
\hline
\tiny $\be_{166}=$ \normalsize 
& \tiny $20,28,29,33,34,36,37 $ \normalsize \\ 
\cline{2-2}
\tiny $\tfrac {1} {660} (-104,-104,-24,116,116,-165,-65,75,155)$ \normalsize 
& \tiny $30,38,39,40 $ \normalsize \\ 
\hline 
\hline
\tiny $\be_{167}=$ \normalsize 
& \tiny $29,30,34,37,38 $ \normalsize \\ 
\cline{2-2}
\tiny $\tfrac {1} {220} (-48,-48,12,12,72,-55,-55,25,85)$ \normalsize 
& \tiny $39,40 $ \normalsize \\
\hline 
\end{tabular}

\end{center}

\begin{center}

\begin{tabular}{|l|l|}
\hline
$\; $\hskip 1in \tiny $\beta$ \hskip 1in $\;$ \normalsize
& $\;$ \hskip 0.9in \tiny $i$ such that 
$\coorde_i\in Z_{\be}$ \hskip 0.9in $\;$ \normalsize \\
\cline{2-2}
& $\;$ \hskip 0.9in \tiny $i$ such that 
$\coorde_i\in W_{\be}$ \hskip 0.9in $\;$ \normalsize \\
\hline 
\hline
\tiny $\be_{185}=$ \normalsize 
& \tiny $10,18,19,28,29,36,37 $ \normalsize \\ 
\cline{2-2}
\tiny $\tfrac {1} {660} (-264,-24,56,116,116,-65,-5,-5,75)$ \normalsize 
& \tiny $20,30,38,39,40 $ \normalsize \\ 
\hline 
\hline 
\tiny $\be_{186}=$ \normalsize 
& \tiny $7,9,16,18,26,28,34 $ \normalsize \\ 
\cline{2-2}
\tiny $\tfrac {1} {180} (-20,-8,-8,16,20,-5,-1,-1,7)$ \normalsize 
& \tiny $10,17,19,20,27,29,30,36,37,38,39,40 $ \normalsize \\ 
\hline 
\hline
\tiny $\be_{187}=$ \normalsize 
& \tiny $19,20,28,34,37 $ \normalsize \\ 
\cline{2-2}
\tiny $\tfrac {1} {820} (-148,-148,52,52,192,-205,-45,95,155)$ \normalsize 
& \tiny $29,30,38,39,40 $ \normalsize \\
\hline 
\end{tabular}

\end{center}

\begin{center}

\begin{tabular}{|l|l|}
\hline
$\; $\hskip 1in \tiny $\beta$ \hskip 1in $\;$ \normalsize
& $\;$ \hskip 1in \tiny $i$ such that 
$\coorde_i\in Z_{\be}$ \hskip 1in $\;$ \normalsize \\
\cline{2-2}
& $\;$ \hskip 1in \tiny $i$ such that 
$\coorde_i\in W_{\be}$ \hskip 1in $\;$ \normalsize \\
\hline 
\hline
\tiny $\be_{205}=$ \normalsize 
& \tiny $10,20,23,24,35 $ \normalsize \\ 
\cline{2-2}
\tiny $\tfrac {1} {380} (-32,-12,-12,28,28,-35,-35,25,45)$ \normalsize 
& \tiny $26,27,28,29,30,33,34,36,37,38,39,40 $ \normalsize \\ 
\hline 
\hline
\tiny $\be_{206}=$ \normalsize 
& \tiny $9,10,18,24,35,36 $ \normalsize \\ 
\cline{2-2}
\tiny $\tfrac {1} {260} (-24,-20,12,12,20,-21,-13,15,19)$ \normalsize 
& \tiny $19,20,27,28,29,30,34,37,38,39,40 $ \normalsize \\ 
\hline 
\hline
\tiny $\be_{207}=$ \normalsize 
& \tiny $10,20,25,31,32 $ \normalsize \\ 
\cline{2-2}
\tiny $\tfrac {1} {1220} (-88,12,12,32,32,-45,-45,-5,95)$ \normalsize 
& \tiny $26,27,28,29,30,33,34,35,36,37,38,39,40 $ \normalsize \\
\hline 
\end{tabular}

\end{center}

\begin{center}

\begin{tabular}{|l|l|}
\hline
$\; $\hskip 1in \tiny $\beta$ \hskip 1in $\;$ \normalsize
& $\;$ \hskip 0.9in \tiny $i$ such that 
$\coorde_i\in Z_{\be}$ \hskip 0.9in $\;$ \normalsize \\
\cline{2-2}
& $\;$ \hskip 0.9in \tiny $i$ such that 
$\coorde_i\in W_{\be}$ \hskip 0.9in $\;$ \normalsize \\
\hline 
\hline
\tiny $\be_{225}=$ \normalsize 
& \tiny $10,19,29,34,36 $ \normalsize \\ 
\cline{2-2}
\tiny $\tfrac {1} {60} (-16,-4,0,4,16,-7,-3,-3,13)$ \normalsize 
& \tiny $20,30,37,38,39,40 $ \normalsize \\ 
\hline 
\hline
\tiny $\be_{226}=$ \normalsize 
& \tiny $9,17,18,24,26,34,36 $ \normalsize \\ 
\cline{2-2}
\tiny $\tfrac {1} {420} (-48,-28,-8,32,52,-25,-5,15,15)$ \normalsize 
& \tiny $10,19,20,27,28,29,30,37,38,39,40 $ \normalsize \\ 
\hline 
\hline
\tiny $\be_{227}=$ \normalsize 
& \tiny $9,10,17,28,34,35,36 $ \normalsize \\ 
\cline{2-2}
\tiny $\tfrac {1} {140} (-26,-6,4,4,24,-15,-5,5,15)$ \normalsize 
& \tiny $19,20,27,29,30,37,38,39,40 $ \normalsize \\
\hline 
\end{tabular}

\end{center}

\begin{center}

\begin{tabular}{|l|l|}
\hline
$\; $\hskip 1in \tiny $\beta$ \hskip 1in $\;$ \normalsize
& $\;$ \hskip 1in \tiny $i$ such that 
$\coorde_i\in Z_{\be}$ \hskip 1in $\;$ \normalsize \\
\cline{2-2}
& $\;$ \hskip 1in \tiny $i$ such that 
$\coorde_i\in W_{\be}$ \hskip 1in $\;$ \normalsize \\
\hline 
\hline
\tiny $\be_{245}=$ \normalsize 
& \tiny $24,27,29,30,34,37,39,40 $ \normalsize \\ 
\cline{2-2}
\tiny $\tfrac {1} {20} (-3,-3,-3,-3,12,-5,-5,5,5)$ \normalsize 
& \tiny $\;$ \hskip 0.3in - \normalsize \\ 
\hline 
\hline
\tiny $\be_{246}=$ \normalsize 
& \tiny $29,30,37,38 $ \normalsize \\ 
\cline{2-2}
\tiny $\tfrac {1} {20} (-8,-2,2,2,6,-5,-5,3,7)$ \normalsize 
& \tiny $39,40 $ \normalsize \\ 
\hline 
\hline
\tiny $\be_{247}=$ \normalsize 
& \tiny $20,30,37,39 $ \normalsize \\ 
\cline{2-2}
\tiny $\tfrac {1} {20} (-8,-3,-3,2,12,-5,0,0,5)$ \normalsize 
& \tiny $40 $ \normalsize \\
\hline 
\end{tabular}

\end{center}

\begin{center}

\begin{tabular}{|l|l|}
\hline
$\; $\hskip 1in \tiny $\beta$ \hskip 1in $\;$ \normalsize
& $\;$ \hskip 1in \tiny $i$ such that 
$\coorde_i\in Z_{\be}$ \hskip 1in $\;$ \normalsize \\
\cline{2-2}
& $\;$ \hskip 1in \tiny $i$ such that 
$\coorde_i\in W_{\be}$ \hskip 1in $\;$ \normalsize \\
\hline 
\hline
\tiny $\be_{265}=$ \normalsize 
& \tiny $10,20,27,29,37,39 $ \normalsize \\ 
\cline{2-2}
\tiny $\tfrac {1} {20} (-8,-2,-2,0,12,-1,-1,1,1)$ \normalsize 
& \tiny $30,40 $ \normalsize \\ 
\hline 
\hline
\tiny $\be_{266}=$ \normalsize 
& \tiny $10,20,28,29,38,39 $ \normalsize \\ 
\cline{2-2}
\tiny $\tfrac {1} {20} (-8,-8,4,6,6,-1,-1,1,1)$ \normalsize 
& \tiny $30,40 $ \normalsize \\ 
\hline 
\hline
\tiny $\be_{267}=$ \normalsize 
& \tiny $10,19,29,36 $ \normalsize \\ 
\cline{2-2}
\tiny $\tfrac {1} {260} (-104,-4,16,36,56,-25,-5,-5,35)$ \normalsize 
& \tiny $20,30,37,38,39,40 $ \normalsize \\
\hline 
\end{tabular}

\end{center}

\begin{center}

\begin{tabular}{|l|l|}
\hline
$\; $\hskip 1in \tiny $\beta$ \hskip 1in $\;$ \normalsize
& $\;$ \hskip 1in \tiny $i$ such that 
$\coorde_i\in Z_{\be}$ \hskip 1in $\;$ \normalsize \\
\cline{2-2}
& $\;$ \hskip 1in \tiny $i$ such that 
$\coorde_i\in W_{\be}$ \hskip 1in $\;$ \normalsize \\
\hline 
\hline
\tiny $\be_{285}=$ \normalsize 
& \tiny $17,19,20,27,29,30,37,39,40 $ \normalsize \\ 
\cline{2-2}
\tiny $\tfrac {1} {60} (-24,-4,-4,-4,36,-15,5,5,5)$ \normalsize 
& \tiny $\;$ \hskip 0.3in - \normalsize \\ 
\hline 
\hline
\tiny $\be_{286}=$ \normalsize 
& \tiny $18,19,20,28,29,30,38,39,40 $ \normalsize \\ 
\cline{2-2}
\tiny $\tfrac {1} {60} (-24,-24,16,16,16,-15,5,5,5)$ \normalsize 
& \tiny $\;$ \hskip 0.3in - \normalsize \\ 
\hline 
\hline
\tiny $\be_{287}=$ \normalsize 
& \tiny $14,17,24,27,38 $ \normalsize \\ 
\cline{2-2}
\tiny $\tfrac {1} {20} (-2,-2,0,0,4,-5,1,1,3)$ \normalsize 
& \tiny $19,20,29,30,34,37,39,40 $ \normalsize \\ 
\hline 
\hline
\tiny $\be_{288}=$ \normalsize 
& \tiny $39,40 $ \normalsize \\ 
\cline{2-2}
\tiny $\tfrac {1} {20} (-8,-8,2,2,12,-5,-5,-5,15)$ \normalsize 
& \tiny $\;$ \hskip 0.3in - \normalsize \\ 
\hline 
\hline
\tiny $\be_{289}=$ \normalsize 
& \tiny $35,36,37,38,39,40 $ \normalsize \\ 
\cline{2-2}
\tiny $\tfrac {1} {20} (-8,2,2,2,2,-5,-5,-5,15)$ \normalsize 
& \tiny $\;$ \hskip 0.3in - \normalsize \\ 
\hline 
\hline
\tiny $\be_{290}=$ \normalsize 
& \tiny $30,40 $ \normalsize \\ 
\cline{2-2}
\tiny $\tfrac {1} {20} (-8,-8,-8,12,12,-5,-5,5,5)$ \normalsize 
& \tiny $\;$ \hskip 0.3in - \normalsize \\ 
\hline 
\hline
\tiny $\be_{291}=$ \normalsize 
& \tiny $29,30,39,40 $ \normalsize \\ 
\cline{2-2}
\tiny $\tfrac {1} {20} (-8,-8,2,2,12,-5,-5,5,5)$ \normalsize 
& \tiny $\;$ \hskip 0.3in - \normalsize \\ 
\hline 
\hline
\tiny $\be_{292}=$ \normalsize 
& \tiny $40 $ \normalsize \\ 
\cline{2-2}
\tiny $\tfrac {1} {20} (-8,-8,-8,12,12,-5,-5,-5,15)$ \normalsize 
& \tiny $\;$ \hskip 0.3in - \normalsize \\ 
\hline 
\end{tabular}

\end{center}

\section{Output for the case (4)}
\label{sec:output4}

In this section we list the output of our programming
for the case (4).  
Let $\gt^*_+$ be the Weyl chamber and 
$\mathbbm a_1\ccd \mathbbm a_{56}$ the coordinate 
vectors both defined in Section \ref{sec:outline-program}. 
The set $\gB$ consists of $\be_i$'s in the following table.

\vskip 10pt

\begin{center}

\begin{tabular}{|l|l|}
\hline
$\; $\hskip 1in \tiny $\beta$ \hskip 1in $\;$ \normalsize
& $\;$ \hskip 1in \tiny $i$ such that 
$\coorde_i\in Z_{\be}$ \hskip 1in $\;$ \normalsize \\
\cline{2-2}
& $\;$ \hskip 1in \tiny $i$ such that 
$\coorde_i\in W_{\be}$ \hskip 1in $\;$ \normalsize \\
\hline
\hline 
\tiny $\be_1=$ \normalsize
& \tiny 
\begin{math}
\begin{matrix}
12,13,14,15,16,17,18,19,20,21,27,28,29,30,31, \\
32,33,34,35,36,37,38,39,40,41,42,43,44,45,46 \hfill 
\end{matrix}
\end{math}
\normalsize \\ 
\cline{2-2}
\tiny $\tfrac {1} {120} (-5,-5,-5,3,3,3,3,3)$ \normalsize 
& \tiny $47,48,49,50,51,52,53,54,55,56 $ \normalsize \\ 
\hline
\hline 
\tiny $\be_2=$ \normalsize 
& \tiny $15,18,20,21,26,27,28,29,31,32,34,37,38,39,41,42,44 $ \normalsize \\ 
\cline{2-2}
\tiny $\tfrac {1} {376} (-13,-5,-5,3,3,3,3,11)$ \normalsize 
& \tiny $30,33,35,36,40,43,45,46,47,48,49,50,51,52,53,54,55,56 $ \normalsize \\ 
\hline 
\hline
\tiny $\be_3=$ \normalsize 
& \tiny $11,15,18,20,21,26,30,33,35,36,37,38,39,41,42,44,47,48,50,53 $ \normalsize \\ 
\cline{2-2}
\tiny $\tfrac {3} {184} (-7,-7,1,1,1,1,1,9)$ \normalsize 
& \tiny $40,43,45,46,49,51,52,54,55,56 $ \normalsize \\
\hline 
\end{tabular}
\end{center}

\newpage

\begin{center}

\begin{tabular}{|l|l|}
\hline
$\; $\hskip 1in \tiny $\beta$ \hskip 1in $\;$ \normalsize
& $\;$ \hskip 1in \tiny $i$ such that 
$\coorde_i\in Z_{\be}$ \hskip 1in $\;$ \normalsize \\
\cline{2-2}
& $\;$ \hskip 1in \tiny $i$ such that 
$\coorde_i\in W_{\be}$ \hskip 1in $\;$ \normalsize \\
\hline
\hline
\tiny $\be_9=$ \normalsize
& \tiny 
\begin{math}
\begin{matrix}
22,23,24,25,26,27,28,29,30,31,32,33,34,35,36,37,38,39, \\
40,41,42,43,44,45,46,47,48,49,50,51,52,53,54,55,56 \hfill 
\end{matrix}
\end{math}
\normalsize \\ 
\cline{2-2}
\tiny $\tfrac {3} {56} (-7,1,1,1,1,1,1,1)$ \normalsize 
& \tiny $\;$ \hskip 0.3in - \normalsize \\ 
\hline
\hline
\tiny $\be_{10}=$ \normalsize 
& \tiny $19,20,21,34,35,36,44,45,46,50,51,52,53,54,55 $ \normalsize \\ 
\cline{2-2}
\tiny $\tfrac {7} {120} (3,-3,-3,-3,-3,5,5,5)$ \normalsize 
& \tiny $56 $ \normalsize \\ 
\hline 
\hline
\tiny $\be_{11}=$ \normalsize 
& \tiny $21,25,26,29,30,32,33,34,35,39,40,42,43,44,45,48,49,50,51,53,54 $ \normalsize \\ 
\cline{2-2}
\tiny $\tfrac {5} {184} (-9,-1,-1,-1,-1,-1,7,7)$ \normalsize 
& \tiny $36,46,52,55,56 $ \normalsize \\
\hline 
\end{tabular}

\end{center}

\begin{center}

\begin{tabular}{|l|l|}
\hline
$\; $\hskip 1in \tiny $\beta$ \hskip 1in $\;$ \normalsize
& $\;$ \hskip 1in \tiny $i$ such that 
$\coorde_i\in Z_{\be}$ \hskip 1in $\;$ \normalsize \\
\cline{2-2}
& $\;$ \hskip 1in \tiny $i$ such that 
$\coorde_i\in W_{\be}$ \hskip 1in $\;$ \normalsize \\
\hline 
\hline
\tiny $\be_{29}=$ \normalsize 
& \tiny $20,21,35,36,45,46,49,50,53 $ \normalsize \\ 
\cline{2-2}
\tiny $\tfrac {13} {568} (-9,-9,-9,-1,-1,7,7,15)$ \normalsize 
& \tiny $51,52,54,55,56 $ \normalsize \\ 
\hline 
\hline
\tiny $\be_{30}=$ \normalsize 
& \tiny $21,36,39,40,42,43,44,45,47 $ \normalsize \\ 
\cline{2-2}
\tiny $\tfrac {9} {632} (-13,-13,-5,3,3,3,11,11)$ \normalsize 
& \tiny $46,48,49,50,51,52,53,54,55,56 $ \normalsize \\ 
\hline 
\hline
\tiny $\be_{31}=$ \normalsize 
& \tiny $20,21,35,36,40,43,44,47,48 $ \normalsize \\ 
\cline{2-2}
\tiny $\tfrac {11} {952} (-15,-15,-7,1,1,9,9,17)$ \normalsize 
& \tiny $45,46,49,50,51,52,53,54,55,56 $ \normalsize \\
\hline 
\end{tabular}

\end{center}

\begin{center}

\begin{tabular}{|l|l|}
\hline
$\; $\hskip 1in \tiny $\beta$ \hskip 1in $\;$ \normalsize
& $\;$ \hskip 1in \tiny $i$ such that 
$\coorde_i\in Z_{\be}$ \hskip 1in $\;$ \normalsize \\
\cline{2-2}
& $\;$ \hskip 1in \tiny $i$ such that 
$\coorde_i\in W_{\be}$ \hskip 1in $\;$ \normalsize \\
\hline 
\hline
\tiny $\be_{48}=$ \normalsize 
& \tiny $21,36,46,52,53,54 $ \normalsize \\ 
\cline{2-2}
\tiny $\tfrac {7} {88} (-3,-3,-3,-3,1,1,5,5)$ \normalsize 
& \tiny $55,56 $ \normalsize \\ 
\hline 
\hline
\tiny $\be_{49}=$ \normalsize 
& \tiny $16,17,18,19,20,21,31,32,33,34,35,36,37,38,39,40 $ \normalsize \\ 
\cline{2-2}
\tiny $\tfrac {1} {88} (-5,-5,-1,-1,3,3,3,3)$ \normalsize 
& \tiny $41,42,43,44,45,46,47,48,49,50,51,52,53,54,55,56 $ \normalsize \\ 
\hline 
\hline
\tiny $\be_{50}=$ \normalsize 
& \tiny $6,11,15,18,19,26,30,33,34,40,43,44,49,50,53 $ \normalsize \\ 
\cline{2-2}
\tiny $\tfrac {1} {16} (-1,-1,-1,-1,-1,1,1,3)$ \normalsize 
& \tiny $20,21,35,36,45,46,51,52,54,55,56 $ \normalsize \\
\hline 
\end{tabular}

\end{center}

\begin{center}

\begin{tabular}{|l|l|}
\hline
$\; $\hskip 1in \tiny $\beta$ \hskip 1in $\;$ \normalsize
& $\;$ \hskip 1in \tiny $i$ such that 
$\coorde_i\in Z_{\be}$ \hskip 1in $\;$ \normalsize \\
\cline{2-2}
& $\;$ \hskip 1in \tiny $i$ such that 
$\coorde_i\in W_{\be}$ \hskip 1in $\;$ \normalsize \\
\hline 
\hline
\tiny $\be_{68}=$ \normalsize 
& \tiny $20,26,30,33,34,39,42,48 $ \normalsize \\ 
\cline{2-2}
\tiny $\tfrac {1} {280} (-25,-17,-9,-9,-9,-1,31,39)$ \normalsize 
& \tiny $21,35,36,40,43,44,45,46,49,50,51,52,53,54,55,56 $ \normalsize \\ 
\hline 
\hline
\tiny $\be_{69}=$ \normalsize 
& \tiny $21,36,37,38,41,47 $ \normalsize \\ 
\cline{2-2}
\tiny $\tfrac {3} {152} (-7,-7,1,1,1,1,5,5)$ \normalsize 
& \tiny $39,40,42,43,44,45,46,48,49,50,51,52,53,54,55,56 $ \normalsize \\ 
\hline 
\hline
\tiny $\be_{70}=$ \normalsize 
& \tiny $34,35,36,44,45,46,50,51,52,53,54,55 $ \normalsize \\ 
\cline{2-2}
\tiny $\tfrac {1} {24} (-9,-3,-3,-3,-3,7,7,7)$ \normalsize 
& \tiny $56 $ \normalsize \\
\hline 
\end{tabular}

\end{center}

\begin{center}

\begin{tabular}{|l|l|}
\hline
$\; $\hskip 1in \tiny $\beta$ \hskip 1in $\;$ \normalsize
& $\;$ \hskip 1in \tiny $i$ such that 
$\coorde_i\in Z_{\be}$ \hskip 1in $\;$ \normalsize \\
\cline{2-2}
& $\;$ \hskip 1in \tiny $i$ such that 
$\coorde_i\in W_{\be}$ \hskip 1in $\;$ \normalsize \\
\hline 
\hline
\tiny $\be_{88}=$ \normalsize 
& \tiny $15,18,19,30,33,34,38,39,41,42 $ \normalsize \\ 
\cline{2-2}
\tiny $\tfrac {1} {104} (-5,-5,-3,1,1,3,3,5)$ \normalsize 
& \tiny $20,21,35,36,40,43,44,45,46,47,48,49,50,51,52,53,54,55,56 $ \normalsize \\ 
\hline 
\hline
\tiny $\be_{89}=$ \normalsize 
& \tiny $21,36,45,49,50,53 $ \normalsize \\ 
\cline{2-2}
\tiny $\tfrac {1} {56} (-13,-13,-9,-1,-1,7,11,19)$ \normalsize 
& \tiny $46,51,52,54,55,56 $ \normalsize \\ 
\hline 
\hline
\tiny $\be_{90}=$ \normalsize 
& \tiny $21,36,40,43,44,47 $ \normalsize \\ 
\cline{2-2}
\tiny $\tfrac {1} {376} (-69,-69,-29,11,11,27,51,67)$ \normalsize 
& \tiny $45,46,48,49,50,51,52,53,54,55,56 $ \normalsize \\
\hline 
\end{tabular}

\end{center}

\begin{center}

\begin{tabular}{|l|l|}
\hline
$\; $\hskip 1in \tiny $\beta$ \hskip 1in $\;$ \normalsize
& $\;$ \hskip 1in \tiny $i$ such that 
$\coorde_i\in Z_{\be}$ \hskip 1in $\;$ \normalsize \\
\cline{2-2}
& $\;$ \hskip 1in \tiny $i$ such that 
$\coorde_i\in W_{\be}$ \hskip 1in $\;$ \normalsize \\
\hline 
\hline
\tiny $\be_{108}=$ \normalsize 
& \tiny $31,32,33,34,35,36,41,42,43,44,45,46,47,48,49,50,51,52 $ \normalsize \\ 
\cline{2-2}
\tiny $\tfrac {1} {24} (-9,-1,-1,-1,3,3,3,3)$ \normalsize 
& \tiny $53,54,55,56 $ \normalsize \\ 
\hline 
\hline
\tiny $\be_{109}=$ \normalsize
& \tiny $18,20,26,30,32,34,39,41,47 $ \normalsize \\ 
\cline{2-2}
\tiny $\tfrac {1} {72} (-5,-3,-1,-1,1,1,3,5)$ \normalsize 
& \tiny 
\begin{math}
\begin{matrix}
21,33,35,36,40,42,43,44,45,46,48,49,50,51,52,53, \\
54,55,56 \hfill 
\end{matrix}
\end{math}
\normalsize \\ 
\hline 
\hline
\tiny $\be_{110}=$ \normalsize 
& \tiny $36,46,52,55,56 $ \normalsize \\ 
\cline{2-2}
\tiny $\tfrac {1} {40} (-15,-7,-7,-7,-7,-7,25,25)$ \normalsize 
& \tiny $\;$ \hskip 0.3in - \normalsize \\
\hline 
\end{tabular}

\end{center}

\begin{center}

\begin{tabular}{|l|l|}
\hline
$\; $\hskip 1in \tiny $\beta$ \hskip 1in $\;$ \normalsize
& $\;$ \hskip 1in \tiny $i$ such that 
$\coorde_i\in Z_{\be}$ \hskip 1in $\;$ \normalsize \\
\cline{2-2}
& $\;$ \hskip 1in \tiny $i$ such that 
$\coorde_i\in W_{\be}$ \hskip 1in $\;$ \normalsize \\
\hline 
\hline
\tiny $\be_{128}=$ \normalsize 
& \tiny $35,36,45,46,51,52,53 $ \normalsize \\ 
\cline{2-2}
\tiny $\tfrac {1} {136} (-51,-19,-19,-19,-11,37,37,45)$ \normalsize 
& \tiny $54,55,56 $ \normalsize \\ 
\hline 
\hline
\tiny $\be_{129}=$ \normalsize 
& \tiny $21,33,35,43,45,48,50 $ \normalsize \\ 
\cline{2-2}
\tiny $\tfrac {1} {264} (-59,-35,-35,-3,13,13,37,69)$ \normalsize 
& \tiny $36,46,49,51,52,53,54,55,56 $ \normalsize \\ 
\hline 
\hline
\tiny $\be_{130}=$ \normalsize 
& \tiny $47,48,49,50,51,52,53,54,55,56 $ \normalsize \\ 
\cline{2-2}
\tiny $\tfrac {3} {40} (-5,-5,-5,3,3,3,3,3)$ \normalsize 
& \tiny $\;$ \hskip 0.3in - \normalsize \\
\hline 
\end{tabular}

\end{center}

\begin{center}

\begin{tabular}{|l|l|}
\hline
$\; $\hskip 1in \tiny $\beta$ \hskip 1in $\;$ \normalsize
& $\;$ \hskip 1in \tiny $i$ such that 
$\coorde_i\in Z_{\be}$ \hskip 1in $\;$ \normalsize \\
\cline{2-2}
& $\;$ \hskip 1in \tiny $i$ such that 
$\coorde_i\in W_{\be}$ \hskip 1in $\;$ \normalsize \\
\hline 
\hline
\tiny $\be_{148}=$ \normalsize
& \tiny $18,19,31,32,40 $ \normalsize \\ 
\cline{2-2}
\tiny $\tfrac {1} {584} (-35,-27,-11,-11,13,21,21,29)$ \normalsize 
& \tiny 
\begin{math}
\begin{matrix}
20,21,33,34,35,36,41,42,43,44,45,46,47,48,49,50, \\
51,52,53,54,55,56 \hfill 
\end{matrix}
\end{math}
\normalsize \\ 
\hline 
\hline
\tiny $\be_{149}=$ \normalsize
& \tiny $21,25,26,29,30,41,47$ \normalsize \\ 
\cline{2-2}
\tiny $\tfrac {1} {328} (-35,-11,-3,-3,5,5,21,21)$ \normalsize 
& \tiny
\begin{math}
\begin{matrix}
32,33,34,35,36,39,40,42,43,44,45,46,48,49,50,51,52, \\
53,54,55,56 \hfill 
\end{matrix}
\end{math}
\normalsize \\ 
\hline 
\hline
\tiny $\be_{150}=$ \normalsize 
& \tiny $20,21,34,40,43,47,48 $ \normalsize \\ 
\cline{2-2}
\tiny $\tfrac {1} {152} (-25,-17,-9,-1,-1,15,15,23)$ \normalsize 
& \tiny $35,36,44,45,46,49,50,51,52,53,54,55,56 $ \normalsize \\
\hline
\end{tabular}

\end{center}

\begin{center}

\begin{tabular}{|l|l|}
\hline
$\; $\hskip 1in \tiny $\beta$ \hskip 1in $\;$ \normalsize
& $\;$ \hskip 1in \tiny $i$ such that 
$\coorde_i\in Z_{\be}$ \hskip 1in $\;$ \normalsize \\
\cline{2-2}
& $\;$ \hskip 1in \tiny $i$ such that 
$\coorde_i\in W_{\be}$ \hskip 1in $\;$ \normalsize \\
\hline 
\hline
\tiny $\be_{168}=$ \normalsize 
& \tiny $35,36,40,43,44,49,50,53 $ \normalsize \\ 
\cline{2-2}
\tiny $\tfrac {1} {8} (-3,-1,0,0,0,1,1,2)$ \normalsize 
& \tiny $45,46,51,52,54,55,56 $ \normalsize \\ 
\hline 
\hline
\tiny $\be_{169}=$ \normalsize 
& \tiny $21,35,45,48 $ \normalsize \\ 
\cline{2-2}
\tiny $\tfrac {1} {104} (-23,-15,-15,1,1,9,17,25)$ \normalsize 
& \tiny $36,46,49,50,51,52,53,54,55,56 $ \normalsize \\ 
\hline
\hline
\tiny $\be_{170}=$ \normalsize
& \tiny $21,25,26,47 $ \normalsize \\ 
\cline{2-2}
\tiny $\tfrac {1} {136} (-15,-3,-3,1,1,1,9,9)$ \normalsize 
& \tiny 
\begin{math}
\begin{matrix}
29,30,32,33,34,35,36,39,40,42,43,44,45,46,48,49, \\
50,51,52,53,54,55,56 \hfill 
\end{matrix}
\end{math}
\normalsize \\
\hline 
\end{tabular}

\end{center}

\bibliographystyle{plain} 
\bibliography{ref4}

\end{document}